\documentclass[draft]{birkjour}
\usepackage[noadjust]{cite}
\usepackage{url}
\usepackage[T2A]{fontenc}
\usepackage[russian,main=english]{babel}    
\usepackage{substitutefont}
\substitutefont{T2A}{\rmdefault}{Tempora-TLF}
\usepackage{amsmath,amssymb}
\DeclareMathOperator{\sech}{sech}
\usepackage{multirow}
\usepackage{amsfonts}
\usepackage{caption}
\input xy
\xyoption{all}
\newtheorem{thm}{Theorem}[section]
\newtheorem{cor}[thm]{Corollary}
\newtheorem{lem}[thm]{Lemma}
\newtheorem{prop}[thm]{Proposition}
\theoremstyle{definition}
\newtheorem{defn}[thm]{Definition}

\theoremstyle{remark}
\newtheorem{rem}[thm]{Remark}

\newtheorem*{ex}{Example}
\numberwithin{equation}{section}

\newcommand{\Cl}{C \kern -0.1em \ell}  

\DeclareMathOperator{\Tr}{\mathrm{Tr}}

\newcommand{\BF}{\mathbb{F}}
\newcommand{\BZ}{\mathbb{Z}}
\newcommand{\BR}{\mathbb{R}}
\newcommand{\BC}{\mathbb{C}}
\newcommand{\BN}{\mathbb{N}}
\newcommand{\ed}{\end{document}}

\newcommand{\SO}{\mathrm{SO}}

\newcommand{\GO}{\mathrm{GO}}

\newcommand{\SL}{\mathrm{SL}}
\newcommand{\GL}{\mathrm{GL}}
\newcommand{\Sp}{\mathrm{Sp}}
\newcommand{\GSp}{\mathrm{GSp}}
\renewcommand{\U}{\mathrm{U}}

\newcommand{\Spin}{\mathbf{Spin}}

\newcommand{\g}{\mathfrak{g}}
\newcommand{\spin}{\mathfrak{spin}}

\renewcommand{\sl}{\mathfrak{sl}}

\begin{document}

%
%
%
%
%
%
%
%
%
\title[Particles and $p$-adic]{Particles and $p-$adic integrals of $\Spin\left(\frac{1}{2}\right)$: spin Lie group, $\mathcal{R}(\rho,q)-$gamma and $\mathcal{R}(\rho,q)-$ beta functions, ghost and applications}%

\author[M. N. Hounkonnou, F. A. Howard]{Mahouton Norbert Hounkonnou, Francis Atta Howard}

\address{
	University of Abomey-Calavi,\\
	International Chair in Mathematical Physics and Applications\\ 
	(ICMPA--UNESCO Chair),
	072 B.P. 50 \\ 
	Cotonou,   Benin Republic}
\email{norbert.hounkonnou@cipma.uac.bj\\ (with copy to hounkonnou@yahoo.fr),\\ hfrancisatta@ymail.com; Francis$ \_ $atta$ \_ $Howard@cipma.net}

\thanks{This work was completed with the support of the NLAGA project.}

\author[K. Kangni]{Kinvi Kangni}
\address{ 
	University Felix Houphouet Boigny of Cocody, \\
	Department of Mathematics and Computer science,\\ 
	22 BP. 1214, Abidjan 22, C\^{o}te d'Ivoire}
\email{kangnikinvi@yahoo.fr}


\date{September 3, 2020}

\begin{abstract}
	In this work, we address the $p$-adic analogues of the fermion spin Lie algebras and Lie groups. We consider the extension of the fermion spin Lie groups and Lie algebras to the $p-$adic Lie groups and investigate the way to extend their integral to the zeta function as well. We  show that their groups
	are ghost friendly. In addition, we develop the $\mathcal{R}(p,q)-$deformed calculus for the Bernoulli, Volkenborn, Euler and Genocchi polynomials, and establish related  definitions. Finally, we perform a $p-$adic  generalization of beta and gamma functions and exhibit some physical applications.   
\end{abstract}

\maketitle
\section{Introduction}\label{INT}
The concept of $p-$adic analysis has many applications in Mathematics and Physics, especially, in string theory, where it is used to compute the zeta functions. The results of the $p-$adic integral were used in several areas of mathematics;  Ilani \cite{ref48} gave an explicit formula for the number of subgroups of index $p^{n}$ in the principle congruence subgroups of $\SL2(\BC_{p})$ (for odd primes $p$), and for the zeta function associated with the group.
 Kim \cite{ref62} introduced  interesting $p-$adic analogues of the Eulerian polynomials.  Some identities of the Eulerian polynomials in connection with the Genocchi, Euler, and tangent numbers were studied. A symmetric relation between the $q-$extension of the alternating sum of integer powers and the Eulerian polynomials was found.
The $(p, q)-$generalization of the binomial coefficients was a focus of Corcino's work \cite{ref14}.  He also discussed a number of helpful features that are similar to those of the regular and $q-$binomial coefficients. Duran $\textit{et al.}$ \cite{ref15} took into account the $(p, q)-$extensions of the Genocchi, Euler, and Bernoulli polynomials and produced the $(p, q)-$analogues of well-known prior formulae and identities. Milovanovic $\textit{et al.}$ \cite{ref55} proceeded to the integral modification of the generalized Bernstien polynomials and developed a novel generalization of beta functions based on $(p, q)-$numbers. As an application, Sadjang \cite{ref59} provided two $(p, q)-$Taylor formulas for polynomials and derived various properties of $(p, q)-$derivatives and $(p,q)-$integrals. Many mathematicians and physicists have studied and investigated special polynomials covering the classical Bernoulli, Euler, and Genocchi polynomials and their generalizations with several applications \cite{ref2,ref7, ref8, ref12, ref13, ref16, ref17, ref58}.
Based on $(\rho, q)-$numbers, Duran $\textit{et al.}$ \cite{ref65} generalized the $p$-adic factorial function and $p-$adic gamma function. They created several recurrence relations and identities using these generalizations. They  developed a number of fresh and intriguing identities and formulae employing some properties of $(\rho, q)-$numbers. In addition, they investigated the $(\rho, q)-$extension of the $p-$adic beta function via the $p-$adic $(\rho, q)-$gamma function. 
In order to generalize known deformed derivatives and integrations of analytical functions defined on a complex disc as special cases corresponding to conveniently selected meromorphic functions, Hounkonnou $\textit{et al.}$, developed a framework for $\mathcal{R}(p,q)-$deformed calculus. This framework provided a method of computation for deformed $\mathcal{R}(p,q)-$derivative and integration. They defined the $\mathcal{R}(p,q)-$derivative and integration and provided pertinent examples under predetermined conditions. More results on this can be found in \cite{ ref34, ref35, ref36, ref37, ref38, ref39, ref40, ref41, ref42, ref43, ref44, ref45}.
Elementary spin particles have Lie structure which are parastatistics elements with some kind of Hopf alagbras. These Lie algebras have its corresponding Lie groups which are specifically spin Lie groups, that is, Fermion spin Lie group and Boson spin Lie group \cite{ref33}. A recent study in \cite{ref33} by Hounkonnou, Howard and Kangni, showed that these spin Lie groups arise from Clifford algebras and they are connected and semisimple. These authors further showed that any spin Lie group $\mathrm{G}$ can be decomposed into  
$$
\mathrm{G}= \mbox{\CYRZH} K D^{s} N
$$ 
where $K$ is compact, $D^{s}$ is a rotational function ($d$-function), and $N$ is nilpotent (Ladder operators), and  \CYRZHDSC $(\alpha^{-1})$ denotes the fine structure constant and all other translational energy of elementary spin particles.
Most inspiringly, their paper revealed that the Lie algebra $\spin (j)$ of a particles can be represented by classical matrices, which make it easier to see their algebraic nature \cite{ref15, ref17, ref22}: 
\[ 
\spin(j)=
\begin{cases}
	\text{higgs}    & \text{$j=0$;}\\
	\text{fermions} & \text{$j=\frac{\mathbb{Z}}{2}$ when odd integer spins are considered};\\ 
	\text{bosons}   & \text{$j=\BZ$ when positive integer spins are considered}.
\end{cases}
\]	
For fermions, when the quantum $j=\frac{\mathbb{Z}}{2},$ that is, odd integer spins are considered, the following questions naturally arise:
\begin{enumerate}\label{NQ}
	\item [(1)] Are fermion spin Lie groups $p-$adic \cite{ref46}?
	\item [(2)] What are the Iwasawa algebras of the $\spin(\frac{1}{2})$?
	\item [(3)] Is there a natural deformation for $\spin(\frac{1}{2})$ ?
	\item [(4)] Can one define the $p-$adic zeta function for $\spin(\frac{1}{2})$?
	
	\item [(5)] Can one extend the $p-$adic integral (fermionic and bosonic due to T. Kim \cite{ref62}) quantum calculus to $\mathcal{R}(p,q)-$deformation developed by Hounkonnou?	
	
\end{enumerate}

Motivated by the work described above, we demonstrate in this study that the spin particle fermionic Lie group structure admits a $p-$adic Lie group structure. We next extend the findings to the $p-$adic zeta integral, and  the known $(p,q)-$calculus results to the $\mathcal{R}(p,q)-$quantum calculus.

The  paper is organized as follows. In Section~\ref{prelim}, we recall main  definitions  and known results  useful in the sequel, and set the notation. 
Section~\ref{HNK} deals with the Hounkonnou $\mathcal{R}(p,q)-$deformed quantum algebras. In Section~\ref{Gamma}, we develop the $\mathcal{R}(p,q)-$gamma and beta functions and some properties. We also set the definition of the $p-$adic spin Lie group and its connections with zeta functions in Section~\ref{Adic}. Further, we consider some applications of the $\mathcal{R}(p,q)-$polynomials in \ref{Poly}. Finally, we end with some concluding remarks in Section~\ref{rem}.
\section{Preliminaries}\label{prelim}
In this section, for the sake of clarity of our presentation, we give a quick overview of adopted notations,  definitions and some properties on the $p-$adic numbers, ultrametric, one parameter subgroup, solvable and nilpotent Lie algebras, unipotent group, finite $p-$groups, ghost polynomials, $p-$adic analytic manifolds, $p-$adic $q-$integrals, Bernoulli, Euler  and Genocchi polynomials. More details can be found in \cite{ref3, ref4, ref5, ref6, ref7, ref8, ref12, ref14, ref17, ref18, ref19, ref20, ref22, ref23, ref24, ref25, ref26, ref27, ref29, ref99, ref93, ref53, ref52, ref49, ref33, ref62, ref69, ref83, ref56, ref70, ref79}. 
\subsection{$p$-adic numbers}
$p$ will denote an arbitrary, but fixed, prime number. Each rational number $x\neq 0$ can be written uniquely as
\begin{align*}
	x=p^{n}\cdotp \dfrac{a}{b}
\end{align*}
with $a,b,n \in \BZ$, $b>0$ , $\gcd(a,b)=1$ and $p \nmid ab$. We put
$$v_{p}(x)=n, \quad \mid x \mid_{p}=p^{-n};$$
here $\mid\cdotp \mid_{p}$ is the $p$-adic absolute value of $\mathbb{Q}$.
This absolute value induces a metric on $\mathbb{Q}$, and the completion of $\mathbb{Q}$ with respect to this metric is the $p$-adic field $\mathbb{Q}_{p}$. Each element of $\BZ_{p}$ is the limit of a Cauchy sequence in $\mathbb{Q}$ whose terms all lie in $\BZ$. 
It follows that each $p$-adic integer is the sum of a series $\sum_{n=0}^{\infty}a_{n}p^{n}$ with $a_{n}\in \BZ$.
Thus, 
\begin{align*}
	\mid x \mid\in \left\lbrace p^{k}: k\in \mathbb{Z} \right\rbrace  \cup \left \lbrace 0 \right \rbrace,
	\mid x+y \mid \leq \max \left\lbrace\mid x \mid, \mid y \mid \right\rbrace 
	\mid xy \mid = \mid x \mid \mid y \mid, \mid 1 \mid=1
\end{align*}
for all $x,y\in \mathbb{Q}_{p}$. For $t>0$, let $ \mathbb{Q}_{p}=\left \lbrace x\in \mathbb{Q}_{p}: \mid x \mid\leq t \right \rbrace. $  
Assume $\mathbb{Q}_{p}$ is an additive abelian locally compact topological group. The set $t$ is then a compact open subgroup of $\mathbb{Q}_{p}$ for each $t>0$.
Assume $\mathbb{\BZ_{p}}$ is the group of $p$-adic integers, and $\BZ$ is a dense subgroup of $\BZ_{p}$.
The proper quotient rings of $\BZ_{p}$ are the familiar finite rings $\BZ_{p}/p^{n} \BZ_{p}\cong \BZ/p^{n} \BZ_{p}. $ $\BZ_{p}$ can be regarded as the inverse limit of $\BZ_{p}/p^{n} \BZ_{p}$. We also put $\mathbb{Q}_{p}^{*}=\mathbb{Q}_{p}-\left \lbrace 0 \right \rbrace$ as the locally compact multiplicative group of $p$-adic numbers.   
The groups $\GL(n,\mathbb{Q}_{p})$, $\SL(n,\mathbb{Q}_{p})$, $\Sp(n,\mathbb{Q}_{p})$ are the $p$-adic analytic groups, taking as open compact subgroups, the groups $\GL(n,\mathbb{Z}_{p})$, $\SL(n,\mathbb{Z}_{p})$, $\Sp(n,\mathbb{Z}_{p})$.

\begin{defn}\cite{ref46}
	The space $(X,d)$ is called	ultrametric space if it satisfies the strong inequality $$  d(x,y)\leq \max d(x,z)+d(z,y)$$
	for all $x,y$ and $z\in X$. 
\end{defn}

\subsection{One-parameter subgroup}
Let $\varphi : \mathbb{Q}_{s} \longrightarrow G_{s}$ be a continuous one-to-one homomophism from the additive group $\mathbb{Q}_{s}$ into $G_{s}$. 
All one-parameter subgroups of $\GL(n,\mathbb{Z}_{p})$ can be described using the exponential map on matrices \cite{ref83}. 

\subsubsection{Solvable and nilpotent Lie algebras}
\begin{defn}\cite{ref82}
	A Lie algebra $g$ is solvable if its derived series terminates that is, $g^{n}=0$ for some $n.$ 
	We say $g$ is nilpotent if the lower central series terminates, that is, $g_{n}=0$ for some $n.$ The Lie algebra $g$ is nilpotent of class at most $n$, if $g_{n}=0$.
\end{defn}
\begin{cor}
	\begin{enumerate}
		\item [(i)] Every nilpotent Lie algebra is Solvable.
		\item [(ii)] The commutator series for the Lie algebra is given by $$g^{0}=g, \quad g^{1}=[g,g],\quad g^{n+1}=[g^{n}, g^{n}].$$
		\item [(iii)] The lower central series is given by
		$$g_{0}=g, \quad g_{1}=[g,g], \quad g_{n+1}=[g_{n},g].$$
		
	\end{enumerate}			
\end{cor}

\begin{ex}
	\begin{enumerate}
		
		\item [(1)] The Lie algebra 
		\[
		g=\left(\begin{array}{ccccc}
			a_{1} &  &  &  & b\\
			& \cdot\\
			&  & \cdot\\
			&  &  & \cdot\\
			0 &  &  &  & a_{n}
		\end{array}\right)
		\] is solvable.
		\bigskip
		$g$ is a set of upper triangular matrices with real diagonal entries such that $a_{i}\in \mathbb{C}$ for $1\leq i \leq n$.
		
		\item [(2)] The Lie algebra 
		\[
		g=\left(\begin{array}{ccccc}
			0 &  &  &  & a\\
			& \cdot\\
			&  & \cdot\\
			&  &  & \cdot\\
			0 &  &  &  & 0
		\end{array}\right)
		\]

		is nilpotent.
		This is also know as the Heisenberg Lie algebra.
	\end{enumerate}
\end{ex}
\subsubsection{Unipotent group}
Let $k$ be a finite field of characteristic $p$ and $V=k^{n}$. The group $\GL(n,k)$ of all invertible $n \times n$ matrices over $k$ may be identified with the group $\GL(V)$ of all $k$-linear automorphisms of $V$. We denote by $\U(n)$ the subgroup consisting of upper uni-triangular matrices\cite{ref27}. An automorphism $g$ of $V$ is unipotent if $(g -1)^{n}$ is the zero endomorphism. A subgroup $H$ of $\GL(n,k)$ is said to be unipotent if each of its elements is unipotent \cite{ref27}.
\begin{rem}
	The group $G$ is a finite $p$-group if $|G| = p^{n}$
	for some $n$ and some prime $p$.
\end{rem}

\begin{thm}\cite{ref27}
	Every finite $p$-group is nilpotent.
\end{thm}

\begin{thm}[ Ghost polynomials]\label{thmG}\cite{ref99}
	Let $G$ be one of the classical groups $\GO_{2l+1}$, $\GSp_{2l}$ or $\GO_{2l}^{+}$ of type $B_{l}$, $C_{l}$ or $D_{l}$ respectively. Then $\BZ_{G}(s)$ has abscissa of convergence $a_{l}+1$ and has natural boundary at $\Re(s)=\beta$ where
	\begin{enumerate}
		\item [(i)] $\beta=l^{2}-1=a_{l}-1$ if $G=\GO_{2l+1}$
		\item [(ii)] $\dfrac{l(l+1)}{2}-2= \dfrac{a_{l}-2}{2}+1$ if $G=\GSp_{2l}$ and 
		\item [(iii)] $\dfrac{l(l-1)}{2}-2= \dfrac{a_{l}-2}{2}$ if $\GO_{2l}^{+}.$
	\end{enumerate}
	
\end{thm}

\begin{thm}\label{thm1}\cite{ref33}
	Any spin Lie group $\Spin(J)$ of a spin particle is:
	\begin{enumerate}
		\item [(i)] connected;
		\item [(ii)] semi-simple if and only if its simple roots are one of the Dynkin's root systems $\Pi(B_{n})$ or $\Pi(D_{n})$ associated with the classical groups $\SO(2n+1)$ and $\SO(2n)$, respectively.
	\end{enumerate}
\end{thm}	

\begin{thm}\label{thm4}\cite{ref33}
	For any $\spin(\frac{2n-1}{2})$, the quantum state of the particle is spanned by $2n$ states and there exists orthogonal matrix element $S_{k_{n}}$ in the $S_{y_{n}}$ matrix which can be transformed into the classical group $\SO(2n)$ with natural numbers $n=1,2,3,\ldots$. This compact Lie group $\SO(2n)$ corresponds to the Dynkin's root $\Pi (D_{n})$. 
\end{thm}

\subsection{Kim $p-$adic $q-$integral}
Let $p$ be a fixed odd prime. Throughout the paper, $\BZ_p$, $\mathbb{Q}_p$, $\BC$, and $\BC_p$ stand for the ring of $p$-adic rational integers, the field of $p$-adic rational numbers, the complex number field, and the completion of the algebraic closure of $\mathbb{Q}_p$, respectively \cite{ref62}. Let $v_{p}$ be the normalized exponential valuation of $\BC_p$ with $|p|p=p-vp(p)=\frac{1}{p}$ .
We assume that $$|q-1|_{p}<P^{-\frac{1}{p-1}},$$ so that $q^{x}=\exp(x\log q)$ for $|x|_{p}\leq 1$. When one speaks of $q$-extension, $q$ can be regarded, depending on the context, as an indeterminate, a complex number $q \in \BC$, or a $p$-adic number $q \in \BC_p$. 
We use the notation $$[x]_q=[x:q] = \dfrac{1 - q^{x}}{1 - q}$$ and $$[x]_-q = \dfrac{1 - (-q)^{x}}{1 + q}.$$ 

For $d$ a fixed positive integer with $(d,p)=1,$ let $$X=X_{d}=\varprojlim{\BZ}/{dp^{N}}, \quad \quad X_{1}=\BZ_{p},$$
$$X^{*}=\bigcup_{0<a<dp}     a+ dp\BZ_{p},  \quad \quad \mbox{with}\quad \mbox{(a,p)=1}$$
$$a+ dp\BZ_{p}= \left\lbrace x\in X| x\equiv a \;(\mod dp^{N}) \right\rbrace $$
where $a\in \BZ$ lies in $0\leq a<dp^{N}$.
For any positive integer $d$ we can obtain
\begin{align}
	\dfrac{1}{[p:q^{dp^{N}}]}\sum_{i=0}^{p-1}q^{idp^{N}}=1.
	\label{9.0}
\end{align}
We can define a $q$-analogue $\mu_{q}$ of basic distribution $\mu_{1}$.
We set 
\begin{align*}
	\mu_{q}(a+ dp\BZ_{p})=\dfrac{q^{a}}{[dp^{N}]}=\dfrac{q^{a}}{[dp^{N}:q]}
\end{align*}
and extend to any distribution $X$. For a $p$-adic distribution $\mu_{0}$ defined by 
$$\mu_{1}(a+ dp\BZ_{p})=\dfrac{1}{[dp^{N}]},$$  we note that $$\lim_{q\rightarrow 1} \mu_{q}=\mu_{1}.$$
To check if it is a distribution on $X,$ 
\begin{align*}
	\sum_{i=0}^{p-1}\mu_{q}(a+ i dp\BZ_{p}+ dp^{N+1}\BZ_{p})=\mu_{q}(a+ dp\BZ_{p}),
\end{align*}
we compute
\begin{align}
	\sum_{i=0}^{p-1}\dfrac{1}{[dp^{N+1}]}q^{a+idp^{N}}&=\dfrac{1}{[dp^{N+1}]}\sum_{i=0}^{p-1}q^{a+idp^{N}}\\
	\nonumber
	&=\dfrac{q^{a}}{[dp^{N+1}]}\sum_{i=0}^{p-1}q^{idp^{N}},
\end{align}
where $$[dp^{N+1}]=\dfrac{1-q^{dp^{N+1}}}{1-q}=[dp^{N}][p:q^{dp^{N}}].$$   
Therefore, 
\begin{align}
	\sum_{i=0}^{p-1}\mu_{q}(a+ i dp\BZ_{p}+ dp^{N+1}\BZ_{p})&=\dfrac{q^{a}}{[dp^{N+1}]}\sum_{i=0}^{p-1}q^{idp^{N}}
	\\
	\nonumber
	&=\dfrac{q^{a}}{[dp^{N}]}\dfrac{1}{[p:q^{dp^{N}}]}\sum_{i=0}^{p-1}q^{idp^{N}};\\
	\nonumber \mbox{from \quad \eqref{9.0}, \quad we \quad have \quad}
	\dfrac{q^{a}}{[dp^{N}]}
	&=\mu_{q}(a+ dp^{N}\BZ_{p}).
	\label{9.2}
\end{align}
For non-negative integers $m$, the distribution \eqref{9.2} yields an integral for the case $d=1$:
\begin{align}
	\int_{\BZ_{p}}[a]^{m}d\mu_{q}(a)=\lim_{N\rightarrow \infty}\sum_{a=0}^{p^{N}-1}[a]^{m}\dfrac{q^{a}}{[p^{N}]}
	=I_{q}([a]^{m}),
\end{align}
which is convergent.

We say that $f$ is a uniformly differentiable function at a point $a\in \BZ$, which we put as $f\in UD(\BZ_{p})$, if the quotient
$$F_{f(x,y)}=\dfrac{f(x)-f(y)}{x-y}$$ has a limit as $f^{\prime}(a)$ as $(x,y)\rightarrow (a,a).$ For $f\in UD(\BZ_{p})$, the fermionic $p$-adic $q$ integral on $\BZ_{p}$ is defined by \cite{ref62} 
\begin{align}
	I_{-q}(f)= \int_{\BZ_{p}} f(x)d \mu_{-q}(x)= \lim_{N\rightarrow \infty}\dfrac{1+q}{1+q^{p^{N}}}\sum_{x=0}^{p^{N}-1}f(x) (-q)^{x}.
\end{align}
For $f: \BZ_{p} \longrightarrow \BC_p$, we can also have 
\begin{align}
	I_{q}(f)&= \lim_{N\rightarrow \infty} \dfrac{1}{[P^{N}]_{q}} \sum_{0\leq x\leq p^{N}} q^{x} f(x) \\
	\nonumber 
	&=\lim_{N\rightarrow \infty} \sum_{0\leq x\leq p^{N}} f(x)\mu_{q}(x+p^{N}\BZ_{p})=\int_{\BZ_{p}}f(x)d \mu_{q}(x). 
\end{align}
When $q=-1,$  we have the following:
\begin{align}
	I_{-1}(f)= \lim_{q\rightarrow -1}I_{q}(f) =\lim_{q\rightarrow -1}\int_{\BZ_{p}}f(x)d \mu_{-q}(x)=\int_{\BZ_{p}}f(x)d \mu_{-1}(x).
	\label{9.3}
\end{align}

From equation \eqref{9.3} we obtain;
\begin{align}
	I_{-1}(f)=\int_{\BZ_{p}}f(x)d \mu_{-1}(x)=\lim_{N\rightarrow \infty}\sum_{x=0}^{p^{N}-1}(-1)^{x}f(x).
\end{align}
Now, given $f_{n}(x)=f(x+n),$ we can obtain $f_{1}(x)=f(x+1).$ Then, 
\begin{align*}
	I_{-1}(f_{1})&=-\lim_{N\rightarrow \infty} \sum_{x=0}^{P^{N}-1}f(x)(-1)^{x}+2f(0) \\
	\nonumber
	&=-I_{-1}(f)+2f(0),
\end{align*}

$$I_{-1}(f_{1})+I_{-1}(f)=2f(0).$$

$$\int_{\BZ_{p}}f(x+1)d\mu_{-1}(x)+ \int_{\BZ_{p}}f(x)d\mu_{-1}(x)=2f(0).$$

\subsubsection{$q$-Bernoulli number ($B_{m}(q)\in \BC_p$) \cite{ref62}}
For $B_{m}(q)\in \BC_p,$ we can see, from \eqref{9.2}, that
$$I_{q}([a]^{m})=B_{m}(q)$$ and $$\lim_{q\rightarrow 1} B_{m}(q)=B_{m},$$
where $B_{m}$ is the $m^{th}$ Bernoulli number.
The generating function 

\begin{align*}
	F_{q}(t)=\sum_{m=0}^{\infty}B_{m}(q)\dfrac{t^{m}}{m!}
\end{align*}
is given by
\begin{align}
	F_{q}(t)=\lim_{r\rightarrow \infty}\dfrac{1}{[p^{r}]}\sum_{i=0}^{p^{r}-1}q^{i}e^{[i]t},
\end{align}
satisfying the difference equation $F_{q}(t)=qe^{t}F_{q}(qt)+1-q-t$. For $q=1$, we obtain
\begin{align*}
	F_{q}(t)=\dfrac{\log e^{t}}{e^{t}-1}=\dfrac{t}{e^{t}-1}=e^{Bt}.
\end{align*}
\bigskip
The generalized $q$-Bernoulli number $B_{m,\chi}(q)$ for a primitive Dirichlet character $\chi,$ with conductor $d$ belonging to a set of positive natural numbers, is given as:

\begin{align}
	B_{m,\chi}(q)&=\int_{X}\chi(a)[a]^{m} d\mu_{q}(a)
	\\
	\nonumber
	&=\lim_{N\rightarrow \infty}\sum_{a=0}^{dp^{N-1}}[a]^{m}\chi(a)\dfrac{q^{a}}{[dp^{N}]}=I_{q}([a]^{m}\chi(a)).
\end{align}
Hence, we get $$\lim_{q\rightarrow 1}B_{m,\chi}(q)=B_{m,\chi},$$ where $B_{m,\chi}(q)$ is the $m^{th}$ generalized Bernoulli number.
In variable $x\in \BC_p$ with $|x|_{p}<1$, the $q$-Bernoulli polynomial is given by:
\begin{align}\label{6.11}
	B_{n,q}(x)&=\int_{\BZ_{p}}[x+t]^{n}_{q}d \mu_{q}(t)=\int_{\BZ_{p}}([x]+q^{x}[t])^{n}d \mu_{q}(t)\\
	\nonumber
	&=\sum_{r=0}^{n}\left(\begin{array}{c} n \\ r \end{array}\right)[x]^{n-r}q^{rx}\int_{\BZ_{p}}[t]^{r}d \mu_{q}(t)\\
	\nonumber
	&=\sum_{r=0}^{n}\left(\begin{array}{c} n \\ r \end{array}\right)[x]^{n-r}_{q}q^{rx}B_{r,q}.
\end{align}

\subsubsection{$q-$Volkenborn \cite{ref62}}
For a fermion, the $q$-Volkenborn integral is given as follows:
\begin{align*}
	K_{r,q}=\int_{\BZ_p}[x]^{r}_{q} d\mu_{-q}(x)\quad \quad \quad \mbox{for} \quad r\in\BN.
\end{align*}
Note that
\begin{align*}
	K_{r,q}=[2]_{q}\left( \dfrac{1}{1-q}\right)^{r} \sum_{l=0}^{r}\left(\begin{array}{c} r \\ l \end{array}\right)(-1)^{l}\dfrac{1}{1+q^{l+1}},
\end{align*}
where $\left(\begin{array}{c} r \\ l \end{array}\right)$ is the binomial coefficient.

\subsubsection{Bernoulli, Euler, Genocchi numbers and polynomials \cite{ref90}}
\begin{defn}\label{dfn 2.14}
	Let $n$ be a position integer. Define
	\begin{gather*}
		\mathcal{A}(x,z:p,q)=\sum_{n=0}^{\infty}B_{n}(x:p,q)\dfrac{z^{n}}{[n]_{p,q}!}=\dfrac{z}{e_{p,q}(z)-1}e_{p,q}(xz)\quad \quad (|z| <2\pi),
		\\
		\nonumber
		\mathbb{D}(x,z:p,q)=\sum_{n=0}^{\infty}E_{n}(x:p,q)\dfrac{z^{n}}{[n]_{p,q}!}=\dfrac{[2]_{p,q}}{e_{p,q}(z)+1}e_{p,q}(xz)	\quad \quad (|z|< \pi),\\
		\nonumber
		\mathbb{M}(x,z:p,q)=\sum_{n=0}^{\infty}G_{n}(x:p,q)\dfrac{z^{n}}{[n]_{p,q}!}=\dfrac{[2]_{p,q}z}{e_{p,q}(z)+1}e_{p,q}(xz)	\quad \quad (|z|< \pi),					
	\end{gather*}
	where $B_{n}(x:p,q)$,  $E_{n}(x:p,q)$ and  $G_{n}(x:p,q)$ are the $(p,q)-$Bernoulli polynomials, $(p,q)-$Euler polynomials and $(p,q)-$Genocchi polynomials respectively. 
	For the case where $x=0,$ we obtain 
	$B_{n}(0:p,q)=B_{n}(p,q)$,  $E_{n}(0:p,q)=E_{n}(p,q)$ and  $G_{n}(0:p,q)=G_{n}(p,q),$ which represent $(p,q)-$Bernoulli numbers, $(p,q)-$Euler numbers and $(p,q)-$Genocchi numbers respectively.
\end{defn}

\begin{rem}
	When we set $p=1$, in Definition 2.9,
	 we obtain the $(q)-$ Bernoulli polynomials, $(q)-$Euler polynomials and $(q)-$Genocchi polynomials, respectively.
	The case, where $x=0,$ leads to 
	$B_{n}(0:q)=B_{n}(q)$,  $E_{n}(0:q)=E_{n}(q)$ and  $G_{n}(0:q)=G_{n}(q),$ which represent $(q)-$Bernoulli numbers, $(q)-$Euler numbers and $(q)-$Genocchi numbers, respectively.
	 When we set $p=q \rightarrow 1$, we get the classical case.	
\end{rem}

\subsubsection{Biedenharn-Macfarlane oscillator algebra \cite{ref6}}
The Biedenharn - Macfarlane
$q-$numbers 
\begin{eqnarray*}
	[n]_{q}=\frac{q^n-q^{-n}}{q-q^{-1}}
\end{eqnarray*}
and $q-$factorials:
\begin{eqnarray*}
	[n]!_{q}= \left\{\begin{array}{lr} 1 \quad \mbox{for   } \quad n=0 \quad \\
		\frac{(q;q)_n}{(q-q^{-1})^n} \quad \mbox{for } \quad n\geq
		1, \quad \end{array} \right.
\end{eqnarray*}
generate the following  relevant
properties:
\begin{eqnarray*}
	\,[n]_{q}&=& \sum_{k=0}^{n-1}q^{n-1-2k},\\
	\,[n+m]_{q}&=& q^{-m}\,[n]_{q}+q^n\,[m]_{q}= q^m\,[n]_{q}+q^{-n}\,[m]_{q},\\
	\,[-m]_{q}&=& -[m]_{q},\\
	\,[n-m]_{q}&=& q^{m}\,[n]_{q}-q^{n}[m]_{q}= q^{-m}\,[n]_{q}-q^{-n}\,[m]_{q},\\
	\,[n]_{q}&=& [2]_{q}\,[n-1]_{q}-[n-2]_{q},
\end{eqnarray*}
where  $n$ and $m$ are  integers.

The $q-$ binomial coefficients
\begin{eqnarray*}
	\left[\begin{array}{c} n \\ k \end{array}\right]_{q}=
	\frac{[n]!_{q}}{[k]!_{q}\,[n-k]!_{q}},\quad 0\leq k\leq n;\;\; n\in\mathbb{N}
\end{eqnarray*} 
or
\begin{eqnarray*}
	\left[\begin{array}{c} n \\ k \end{array}\right]_{q}=
	\frac{(q;q)_n}{(q;q)_k(q;,q))_{n-k}},\quad 0\leq k\leq n;\;\; n\in\mathbb{N},
\end{eqnarray*}
where $(q;q)_m = (q-q^{-1})(q^2-q^{-2})\cdots(q^m-q^{-m})$,
$m\in\mathbb{N},$ satisfy the next identities:
\begin{eqnarray*}
	\,\left[\begin{array}{c} n \\ k \end{array}\right]_{q}&=& \left[\begin{array}{c} n \\ n-k \end{array}\right]_{q}=
	q^{k(n-k)}\left[\begin{array}{c} n \\ k \end{array}\right]_{q^{-2}}=
	q^{k(n-k)}\left[\begin{array}{c} n \\ n-k \end{array}\right]_{q^{-2}},\\
	\,\left[\begin{array}{c} n+1 \\ k \end{array}\right]_{q} &=& q^k\left[\begin{array}{c} n \\ k \end{array}\right]_{q}
	+q^{-n-1+k}\left[\begin{array}{c} n \\ k-1 \end{array}\right]_{q},\\
\end{eqnarray*}
\begin{eqnarray*}
	\left[\begin{array}{c} n+1 \\ k \end{array}\right]_{q}&=&
	q^{k}\left[\begin{array}{c} n \\ k \end{array}\right]_{q} +
	q^{n+1-k}\left[\begin{array}{c} n \\ k-1 \end{array}\right]_{q} -(q^n-q^{-n})
	\left[\begin{array}{c} n-1 \\ k-1 \end{array}\right]_{q}
\end{eqnarray*}
with
\begin{eqnarray*}
	\left[\begin{array}{c} n \\ k \end{array}\right]_{q^{-2}}= \frac{(q^{-2}; q^{-2})_n}{(q^{-2}; q^{-2})_k(q^{-2}; q^{-2})_{n-k}},
\end{eqnarray*}
where $(q^{-2}; q^{-2})_n = (1-q^{-2})(1-q^{-4})\cdots (1-q^{-2n})$;
and the $q-$shifted factorial
\begin{eqnarray*}
	((a,b);(q,q^{-1}))_n &:=& (a-b)(aq-bq^{-1})\cdots(aq^{n-1}-bq^{-n+1})
\end{eqnarray*}
or
\begin{eqnarray*}
	((a,b);(q,q^{-1}))_n = \sum_{k=0}^{n}\left[\begin{array}{c} n \\ k \end{array}\right]_{q}(-1)^k\,
	q^{-k(n-k)}a^{n-k}b^k.
\end{eqnarray*}
The generators of the deformed algebra introduced by Biedenharn \cite{ref6} and independently
by Macfarlane \cite{ref57}, in the context of oscillator realization of the quantum
algebra $su_q(2),$ satisfy the following relations:
\begin{eqnarray*}
	A\;A^\dag- q\,A^\dag A= q^{-N}, \quad &&A\;A^\dag- q^{-1}\,A^\dag A= q^N,\qquad 0<q <1\cr
	[N,\;A^\dag]= A^\dag,\quad\qquad\quad&& [N,\;A]= -A.
\end{eqnarray*}

\subsection{Hounkonnou-Bukweli $\mathcal{R}(p,q)-$deformed quantum algebras}\label{HNK}
Let $p$ and $q$ be two positive real numbers such that $ 0<q<p\leq 1.$ We consider a meromorphic function ${\mathcal R}$ defined on $\mathbb{C}\times\mathbb{C}$ by \cite{ref41}: 

\begin{equation}\label{r10}
	\mathcal{R}(u,v)= \sum_{s,t=-l}^{\infty}r_{st}u^sv^t,
\end{equation}
with an eventual isolated singularity at the zero, 
where $r_{st}$ are complex numbers, $l\in\mathbb{N}\cup\left\lbrace 0\right\rbrace,$ $\mathcal{R}(p^n,q^n)>0,  \forall n\in\mathbb{N},$ and $\mathcal{R}(1,1)=0$ by definition. We denote by $\mathbb{D}_{R}$ the bi-disk
\begin{eqnarray*}
	\mathbb{D}_{R}&:=&\prod_{j=1}^{2}\mathbb{D}_{R_j}\nonumber\\
	&=&\left\lbrace w=(w_1,w_2)\in\mathbb{C}^2: |w_j|<R_{j} \right\rbrace,
\end{eqnarray*}
where $R$ is the convergence radius of the series (\ref{r10}) defined by Hadamard formula as follows:
\begin{eqnarray*}
	\lim\sup_{s+t \longrightarrow \infty} \sqrt[s+t]{|r_{st}|R^s_1\,R^t_2}=1.
\end{eqnarray*}
For the proof and more details see \cite{ref41}. Let us also consider $\mathcal{O}(\mathbb{D}_{R})$ the set of holomorphic functions defined on $\mathbb{D}_{R}.$
Define the  $\mathcal{R}(p,q)-$ deformed numbers  \cite{ref41}:
\begin{equation}\label{rpqnumber}
	[n]_{\mathcal{R}(p,q)}:=\mathcal{R}(p^n,q^n),\quad n\in\mathbb{N},
\end{equation}
 the
$\mathcal{R}(p,q)-$ deformed factorials:
\begin{equation*}\label{s0}
	[n]!_{\mathcal{R}(p,q)}:=\left \{
	\begin{array}{l}
		1\quad\mbox{for}\quad n=0\\
		\\
		\mathcal{R}(p,q)\cdots\mathcal{R}(p^n,q^n)\quad\mbox{for}\quad n\geq 1,
	\end{array}
	\right .
\end{equation*}
and the  $\mathcal{R}(p,q)-$ deformed binomial coefficients:
\begin{eqnarray*}\label{bc}
	\bigg[\begin{array}{c} m  \\ n\end{array} \bigg]_{\mathcal{R}(p,q)} := \frac{[m]!_{\mathcal{R}(p,q)}}{[n]!_{\mathcal{R}(p,q)}[m-n]!_{\mathcal{R}(p,q)}},\quad m,n=0,1,2,\cdots,\quad m\geq n
\end{eqnarray*}
satisfying the relation:
\begin{equation*}
	\bigg[\begin{array}{c} m  \\ n\end{array} \bigg]_{\mathcal{R}(p,q)}=\bigg[\begin{array}{c} m  \\ m-n\end{array} \bigg]_{\mathcal{R}(p,q)},\quad m,n=0,1,2,\cdots,\quad m\geq n.
\end{equation*}
Consider the following linear operators defined on  $\mathcal{O}(\mathbb{D}_{R})$ by (see \cite{ref41} for more details):
\begin{eqnarray*}
	\;Q:\varPsi\longmapsto Q\varPsi(z):&=& \varPsi(qz),\\
	\; P:\varPsi\longmapsto P\varPsi(z):&=& \varPsi(pz),
\end{eqnarray*}
and the $\mathcal{R}(p,q)-$ derivative given by:
\begin{equation}\label{r5}
	\partial_{\mathcal{R}( p,q)}:=\partial_{p,q}\frac{p-q}{P-Q}\mathcal{R}( P,Q)=\frac{p-q}{p^{P}-q^{Q}}\mathcal{R}(p^{P},q^{Q})\partial_{p,q}.
\end{equation}
The following conditions can help retrieve some relevant $q-$analogues and $(p,q)-$analogues:
\begin{itemize}\label{HRi}
	\item[(a)] 
	Setting	$\mathcal{R}(1,q)=1$, we obtain the $q-$ Heine derivative \cite{ref32}
	\begin{equation*}
		\partial_{q}\varPsi(z)=\frac{\varPsi(z)-\varPsi(qz)}{z(1-q)}.
	\end{equation*} 
	\item[(b)] Setting $\mathcal{R}(1,q)={1-q^{-1} \over q-1}$ tends to the $q-$ Quesne  derivative \cite{ref79}
	\begin{equation*}
		\partial_{q}\varPsi(z)=\frac{\varPsi(z)-\varPsi(q^{-1}z)}{z(q-1)}.
	\end{equation*}
	\item [(c)]Setting $\mathcal{R}(1,q)=1$ gives the
	$q-$ Biedenharn-Macfarlane derivative \cite{ref6}
	\begin{equation*}
		\partial_{q}\varPsi(z)=\frac{\varPsi(qz)-\varPsi(q^{-1}z)}{z(q-q^{-1})}.
	\end{equation*}
	\item[(d)] Setting $\mathcal{R}(p,q)=1$ affords the
	$(p,q)-$ Jagannathan-Srinivasa derivative \cite{ref75}
	\begin{equation*}\label{J.S}
		\partial_{p,q}\varPsi(z)=\frac{\varPsi(pz)-\varPsi(qz)}{z(p-q)}.
	\end{equation*}
	\item[(e)] Setting $\mathcal{R}(p,q)={1-p\,q \over (p^{-1}-q)p}$ gives the $(p^{-1},q)-$ Chakrabarty - Jagannathan derivative \cite{ref10,ref11}
	\begin{eqnarray*}	
		\partial_{p^{-1},q}\varPsi(z)=\frac{\varPsi(p^{-1}z)-\varPsi(qz)}{z(p^{-1}-q)}.
	\end{eqnarray*}
	\item[(f)]Setting $\mathcal{R}(p,q)={p\,q-1 \over (q-p^{-1})q}$
	we arrive at the Hounkonnou-Ngompe  generalization of $q-$ Quesne derivative \cite{ref44}
	\begin{eqnarray*}
		\partial_{p,q}\varPsi(z)=\frac{\varPsi(pz)-\varPsi(q^{-1}z)}{z(q-p^{-1})}.
	\end{eqnarray*}
\end{itemize}
The  algebra associated with the $\mathcal{R}(p,q)-$ deformation is a quantum algebra, denoted $\mathcal{A}_{\mathcal{R}(p,q)},$ generated by the set of operators $\{1, A, A^{\dagger}, N\}$ satisfying the following commutation relations:
\begin{eqnarray*}
	&& \label{algN1}
	\quad A A^\dag= [N+1]_{\mathcal {R}(p,q)},\quad\quad\quad A^\dag  A = [N]_{\mathcal {R}(p,q)}.
	\cr&&\left[N,\; A\right] = - A, \qquad\qquad\quad \left[N,\;A^\dag\right] = A^\dag
\end{eqnarray*}
with its realization on  ${\mathcal O}(\mathbb{D}_R)$ given by:
\begin{eqnarray*}\label{algNa}
	A^{\dagger} := z,\qquad A:=\partial_{\mathcal {R}(p,q)}, \qquad N:= z\partial_z,
\end{eqnarray*} 
where $\partial_z:=\frac{\partial}{\partial z}$ is the usual derivative on $\mathbb{C}.$
\section{Hounkonnou-Bukweli $\mathcal {R}(p,q)$ exponential functions and trignometric functions}
The Hounkonnou $\mathcal {R}(p,q)$-exponential functions is given by the following;
\begin{enumerate}
	\item [(1)]
	\begin{gather}\label{E4.4}
		\mathbb{E}_{\mathcal {R}(p,q)}(z)=\sum_{n=0}^{\infty}\dfrac{\xi_{2}^{\left(\begin{array}{ccc}n\\2\end{array}\right)}}{[n]_{\mathcal {R}(p,q)}!}z^{n},\quad \mbox{and} \quad
		e_{\mathcal {R}(p,q)}(z)=\sum_{n=0}^{\infty}\dfrac{\xi_{1}^{\left(\begin{array}{ccc}n\\2\end{array}\right)}}{[n]_{\mathcal {R}(p,q)}!}z^{n},
	\end{gather}
	with $\mathbb{E}_{\mathcal {R}(p,q)}(-z)= 	e_{\mathcal {R}(p,q)}(z)=1$.\\
	One can easily retrieve the $(p,q)-$analogue just in the case of \cite{ref75}. We obtain the Jagannathan-Srinivasa $(p,q)-$exponential functions with $\xi_{2}=q$ and $\xi_{1}=p$. If $\lim_{p\rightarrow 1},$ we have the analogue $q$-exponential functions, such as, $\mathbb{E}_{q}(z)$ and $e_{q}(z).$ When
	$\lim_{p,q\longrightarrow 1},$
	we obtain the usual exponential function $e^{z}$. Note that  
	
	\begin{eqnarray*}
		\mathbb{E}_{q}(z)=e_{q^{-1}}(z)\\
		e_{q^{-1}}(z)=\mathbb{E}_{q}(z)\\
		e_{(p,q)^{-1}}(z)\mathbb{E}_{p,q}(-z)=1\\
		e_{(p^{-1},q^{-1})}(z)=\mathbb{E}_{p,q}(z).
	\end{eqnarray*}
	In the case of the complex numbers and trigonometric functions, we get the following identities:	
	\begin{enumerate}
		\item [(1a)] 
		\begin{eqnarray*}
			\mathbb{E}_{\mathcal{R}(p,q)}(iz)&=&\sum_{n=0}^{\infty}\dfrac{\xi_{2}^{\left(\begin{array}{ccc}n\\2\end{array}\right)}(iz)^{n}}{\mathcal{R!}(p^{n},q^{n})}\\ \nonumber
			&=&\sum_{n=0}^{\infty}\dfrac{(-1)^{n}(z)^{2n}\xi_{2}^{(2n-1)n}}{\mathcal{R!}(p^{2n},q^{2n})}+i\sum_{n=0}^{\infty}\dfrac{(-1)^{n}(z)^{2n+1}\xi_{2}^{(2n+1)n}}{\mathcal{R!}(p^{2n+1},q^{2n+1})},
		\end{eqnarray*}
		
		\begin{eqnarray*}
			\mathbb{E}_{\mathcal{R}(p,q)}(-iz)&=&\sum_{n=0}^{\infty}\dfrac{\xi_{2}^{\left(\begin{array}{ccc}n\\2\end{array}\right)}(-iz)^{n}}{\mathcal{R!}(p,q)}\\ \nonumber
			&=&\sum_{n=0}^{\infty}\dfrac{(-1)^{n}(z)^{2n}\xi_{2}^{(2n-1)n}}{\mathcal{R!}(p^{2n},q^{2n})}-i\sum_{n=0}^{\infty}\dfrac{(-1)^{n}(z)^{2n+1}\xi_{2}^{(2n+1)n}}{\mathcal{R!}(p^{2n+1},q^{2n+1})},
		\end{eqnarray*}
		where
		
		\begin{gather*}
			\sum_{n=0}^{\infty}\dfrac{(-1)^{n}(z)^{2n}\xi_{2}^{(2n-1)n}}{\mathcal{R!}(p^{2n},q^{2n})}=\dfrac{\mathbb{E}_{\mathcal{R}(p,q)}^{(iz)}+\mathbb{E}_{\mathcal{R}(p,q)}^{-(iz)}}{2}=\mathbb{COS}_{\mathcal{R}(p,q)}(z),\\ \nonumber	\sum_{n=0}^{\infty}\dfrac{(-1)^{n}(z)^{2n+1}\xi_{2}^{(2n+1)n}}{\mathcal{R!}(p^{2n+1},q^{2n+1})}=\dfrac{\mathbb{E}_{\mathcal{R}(p,q)}^{(iz)}-\mathbb{E}_{\mathcal{R}(p,q)}^{-(iz)}}{2i}=\mathbb{SIN}_{\mathcal{R}(p,q)}(z).
		\end{gather*}
		In the case of the hyperbolic functions, we have:
		\begin{gather*}
			\sum_{n=0}^{\infty}\dfrac{(z)^{2n}\xi_{2}^{(2n-1)n}}{\mathcal{R!}(p^{2n},q^{2n})}=\dfrac{\mathbb{E}_{\mathcal{R}(p,q)}^{(z)}+\mathbb{E}_{\mathcal{R}(p,q)}^{-(z)}}{2}=	\mathbb{COSH}_{\mathcal{R}(p,q)}(z), \\
			\nonumber 	\sum_{n=0}^{\infty}\dfrac{(z)^{2n+1}\xi_{2}^{(2n+1)n}}{\mathcal{R!}(p^{2n+1},q^{2n+1})}=\dfrac{\mathbb{E}_{\mathcal{R}(p,q)}^{(z)}-\mathbb{E}_{\mathcal{R}(p,q)}^{-(z)}}{2i}=	\mathbb{SINH}_{\mathcal{R}(p,q)}(z).
		\end{gather*}
		
		
		As a consequence, we get:
		
		\begin{align*}
			\mathbb{E}_{\mathcal{R}(p,q)}^{(iz)}= \mathbb{COS}_{\mathcal{R}(p,q)}(z)+i \mathbb{SIN}_{\mathcal{R}(p,q)}(z)
		\end{align*}
		
		\item [(1b)] Similarly, from the equation \eqref{E4.4}:
		\begin{eqnarray*}
			e_{\mathcal{R}(p,q)}(iz)&=&\sum_{n=0}^{\infty}\dfrac{\xi_{1}^{\left(\begin{array}{ccc}n\\2\end{array}\right)}(iz)^{n}}{\mathcal{R!}(p^{n},q^{n})}\\ \nonumber
			&=&\sum_{n=0}^{\infty}\dfrac{(-1)^{n}(z)^{2n}\xi_{1}^{(2n-1)n}}{\mathcal{R!}(p^{2n},q^{2n})}+i\sum_{n=0}^{\infty}\dfrac{(-1)^{n}(z)^{2n+1}\xi_{1}^{(2n+1)n}}{\mathcal{R!}(p^{2n+1},q^{2n+1})}
		\end{eqnarray*}
		
		\begin{eqnarray*}
			e_{\mathcal{R}(p,q)}(-iz)&=&\sum_{n=0}^{\infty}\dfrac{\xi_{1}^{\left(\begin{array}{ccc}n\\2\end{array}\right)}(-iz)^{n}}{\mathcal{R!}(p,q)}\\ \nonumber
			&=&\sum_{n=0}^{\infty}\dfrac{(-1)^{n}(z)^{2n}\xi_{1}^{(2n-1)n}}{\mathcal{R!}(p^{2n},q^{2n})}-i\sum_{n=0}^{\infty}\dfrac{(-1)^{n}(z)^{2n+1}\xi_{1}^{(2n+1)n}}{\mathcal{R!}(p^{2n+1},q^{2n+1})}
		\end{eqnarray*}
		where
		
		\begin{align*}
			\sum_{n=0}^{\infty}\dfrac{(-1)^{n}(z)^{2n}\xi_{2}^{(2n-1)n}}{\mathcal{R!}(p^{2n},q^{2n})}=\dfrac{e_{\mathcal{R}(p,q)}^{(iz)}+e_{\mathcal{R}(p,q)}^{-(iz)}}{2}=\cos_{\mathcal{R}(p,q)}(z),\\ \nonumber	\sum_{n=0}^{\infty}\dfrac{(-1)^{n}(z)^{2n+1}\xi_{2}^{(2n+1)n}}{\mathcal{R!}(p^{2n+1},q^{2n+1})}=\dfrac{e_{\mathcal{R}(p,q)}^{(iz)}-e_{\mathcal{R}(p,q)}^{-(iz)}}{2i}=\sin_{\mathcal{R}(p,q)}(z).
		\end{align*}
		In the case of the hyperbolic functions, we have;
		\begin{align*}
			\sum_{n=0}^{\infty}\dfrac{(z)^{2n}\xi_{2}^{(2n-1)n}}{\mathcal{R!}(p^{2n},q^{2n})}=\dfrac{e_{\mathcal{R}(p,q)}^{(z)}+e_{\mathcal{R}(p,q)}^{-(z)}}{2}=	\cosh_{\mathcal{R}(p,q)}(z), \\
			\nonumber 	\sum_{n=0}^{\infty}\dfrac{(z)^{2n+1}\xi_{2}^{(2n+1)n}}{\mathcal{R!}(p^{2n+1},q^{2n+1})}=\dfrac{e_{\mathcal{R}(p,q)}^{(z)}-e_{\mathcal{R}(p,q)}^{-(z)}}{2i}=	\sinh_{\mathcal{R}(p,q)}(z).
		\end{align*}
		
		
		As a consequnce, we obtain the following:
		
		\begin{align*}
			e_{\mathcal{R}(p,q)}^{(iz)}= \cos_{\mathcal{R}(p,q)}(z)+i \sin_{\mathcal{R}(p,q)}(z).
		\end{align*}
		Similarly,
		
		\begin{eqnarray*}
			\tan_{\mathcal{R}(p,q)}(z)=\dfrac{\sin_{\mathcal{R}(p,q)}(z)}{\cos_{\mathcal{R}(p,q)}(z)} \quad \mbox{and} \quad  \mathbb{TAN}_{\mathcal{R}(p,q)}(z)=\dfrac{\mathbb{SIN}_{\mathcal{R}(p,q)}(z)}{\mathbb{COS}_{\mathcal{R}(p,q)}(z)}.
		\end{eqnarray*}
	\end{enumerate}	
\end{enumerate}	
The Euler polynomials with respect to the $\mathcal{R}(p,q)-$trigonometry are defined by: 
\begin{align*}
	\dfrac{[2]_{\mathcal{R}(p,q)}}{e_{\mathcal{R}(p,q)}^{z}+1}e_{\mathcal{R}(p,q)}^{zt}&=\sum_{n=0}^{\infty}E^{p}_{n}\dfrac{z^{n}}{[n]_{\mathcal{R}(p,q)}!} \quad \mbox{and} \\
	\dfrac{[2]_{\mathcal{R}(p,q)}}{\mathbb{E}_{\mathcal{R}(p,q)}^{z}+1}\mathbb{E}_{\mathcal{R}(p,q)}^{zt}&=\sum_{n=0}^{\infty}E^{q}_{n}\dfrac{z^{n}}{[n]_{\mathcal{R}(p,q)}!} 
\end{align*}	
is the $n$-th Euler numbers whereas
\begin{align*}
	\dfrac{[2]_{\mathcal{R}(p,q)}}{e_{\mathcal{R}(p,q)}^{z}+e_{\mathcal{R}(p,q)}^{(-z)}}&=\sum_{n=0}^{\infty}E^{p*}_{n}\dfrac{z^{n}}{[n]_{\mathcal{R}(p,q)}!} \quad \mbox{and} \\
	\dfrac{[2]_{\mathcal{R}(p,q)}}{\mathbb{E}_{\mathcal{R}(p,q)}^{z}+\mathbb{E}_{\mathcal{R}(p,q)}^{(-z)}}&=\sum_{n=0}^{\infty}E^{q*}_{n}\dfrac{z^{n}}{[n]_{\mathcal{R}(p,q)}!}\\
\end{align*}
In particular, 

\begin{eqnarray*}	
	\sech_{\mathcal{R}(p,q)}(z)=\dfrac{1}{\cos_{\mathcal{R}(p,q)}(z)}=\dfrac{[2]_{\mathcal{R}(p,q)}}{e_{\mathcal{R}(p,q)}^{z}+e_{\mathcal{R}(p,q)}^{(-z)}}=\sum_{n=0}^{\infty}E^{*}_{n}\dfrac{z^{n}}{[n]_{\mathcal{R}(p,q)}!} \\ \mathbb{SECH}_{\mathcal{R}(p,q)}(z)=\dfrac{1}{\mathbb{COS}_{\mathcal{R}(p,q)}(z)}=\dfrac{[2]_{\mathcal{R}(p,q)}}{\mathbb{E}_{\mathcal{R}(p,q)}^{z}+\mathbb{E}_{\mathcal{R}(p,q)}^{(-z)}}=\sum_{n=0}^{\infty}E^{*}_{n}\dfrac{z^{n}}{[n]_{\mathcal{R}(p,q)}!}\\
\end{eqnarray*}
where for $n\geq0$, $E^{*}_{n}=2^{n}E^{*}_{n}(\frac{1}{2}).$

\begin{eqnarray*}
	\csc_{\mathcal{R}(p,q)}(z)=\dfrac{1}{\sin_{\mathcal{R}(p,q)}(z)} \quad \mbox{and}\quad \mathbb{CSC}_{\mathcal{R}(p,q)}(z)=\dfrac{1}{\mathbb{SIN}_{\mathcal{R}(p,q)}(z)}
\end{eqnarray*}

\begin{eqnarray*}
	\tanh_{\mathcal{R}(p,q)}(z)=\dfrac{\sinh_{\mathcal{R}(p,q)}(z)}{\cosh_{\mathcal{R}(p,q)}(z)} \quad \mbox{and} \quad  \mathbb{TANH}_{\mathcal{R}(p,q)}(z)=\dfrac{\mathbb{SINH}_{\mathcal{R}(p,q)}(z)}{\mathbb{COSH}_{\mathcal{R}(p,q)}(z)}\\
	\coth_{\mathcal{R}(p,q)}(z)=\dfrac{\cosh_{\mathcal{R}(p,q)}(z)}{\sinh_{\mathcal{R}(p,q)}(z)} \quad \mbox{and} \quad  \mathbb{COTH}_{\mathcal{R}(p,q)}(z)=\dfrac{\mathbb{COSH}_{\mathcal{R}(p,q)}(z)}{\mathbb{SINH}_{\mathcal{R}(p,q)}(z)}	
\end{eqnarray*}

\begin{defn}
	Let $n$ be a non-negative integer. The Hounkonnou \\$\mathcal{R}(p,q)-$analogue of Euler's zigzag numbers are given by
	
	\begin{gather*}
		F(x)=\sum_{n=0}^{\infty}A_{n}\dfrac{x^{n}}{[n]_{\mathcal{R}(p,q)}!},
	\end{gather*}
	with  $$\sum_{n=0}^{\infty}A_{n}\dfrac{x^{n}}{[n]_{\mathcal{R}(p,q)}!}=\sum_{n=0}^{\infty}A_{2n}\dfrac{x^{2n}}{[2n]_{\mathcal{R}(p,q)}!}+\sum_{n=0}^{\infty}A_{2n-1}\dfrac{x^{2n-1}}{[2n-1]_{\mathcal{R}(p,q)}!},$$
	
	where
	\begin{eqnarray*}
		\sec_{\mathcal{R}(p,q)}(x)=\sum_{n=0}^{\infty}A_{2n}\dfrac{x^{2n}}{[2n]_{\mathcal{R}(p,q)}!}=\sum_{n=1}^{\infty}S_{n}(x:\mathcal{R}(p,q)) \quad \quad \mbox{and}\\	 
		\tan_{\mathcal{R}(p,q)}(x)=\sum_{n=0}^{\infty}A_{2n-1}\dfrac{x^{2n-1}}{[2n-1]_{\mathcal{R}(p,q)}!}=\sum_{n=1}^{\infty}T_{n}(x:\mathcal{R}(p,q)),
	\end{eqnarray*}
	where $T_{n}$ is the tangent number.
\end{defn}

\subsection{Hounkonnou-Bukweli $\mathcal{R}(p,q)-$derivative}
It is defined from the following linear operators on  $\mathcal{O}(\mathbb{D}_{R}):$
\begin{eqnarray*}
	\;Q:\varPsi\longmapsto Q\varPsi(z):&=& \varPsi(qz),\\
	\; P:\varPsi\longmapsto P\varPsi(z):&=& \varPsi(pz),\\
	\; \partial_{p,q}:\varPsi\longmapsto \partial_{p,q} \varPsi(z):&=& \dfrac{\varPsi(pz)-\varPsi(qz)}{(p-q)z}.
\end{eqnarray*}
If we set $d_{\mathcal{R}(p,q)}=(dz)\partial_{\mathcal{R}( p,q)}$ and from equation \eqref{algN1}, then  follow the properties:

\begin{enumerate}
	\item [(i)] $d_{\mathcal{R}(p,q)} 1=0$
	\item [(ii)]$d_{\mathcal{R}(p,q)}z=(dz)\mathcal{R}(p,q)$
	\item [(iii)] $d_{\mathcal{R}(p,q)}\partial_{\mathcal{R}( p,q)}=(dz)\partial_{\mathcal{R}(p,q)}^{2}$
	\item [(iv)] $d_{\mathcal{R}(p,q)}(z \partial_{z})=dz(z\partial_{z}+1)\partial_{\mathcal{R}( p,q)}$
	\item [(v)] $d_{\mathcal{R}(p,q)}^{2}=0$	
\end{enumerate} 
For a non-negative integer $n:$
\begin{enumerate}
	\item [(i)]$d_{\mathcal{R}(p,q)}(z^{n})=dz \mathcal{R}(p^{n},q^{n})z^{n-1}$
	\item [(ii)]  $d_{\mathcal{R}(p,q)}(z \partial_{z})^{n}=dz(z\partial_{z}+1)^{n}\partial_{\mathcal{R}( p,q)}$
	\item [(iii)] $d_{\mathcal{R}(p,q)}\partial_{\mathcal{R}( p,q)}^{n}=(dz)\partial_{\mathcal{R}(p,q)}^{n+1}$
\end{enumerate}
\bigskip
The $\mathcal{R}(p,q)-$ derivative is then defined by:
\begin{equation}\label{r5}
	\partial_{\mathcal{R}( p,q)}:=\partial_{p,q}\frac{p-q}{P-Q}\mathcal{R}( P,Q)=\frac{p-q}{p^{P}-q^{Q}}\mathcal{R}(p^{P},q^{Q})\partial_{p,q}
\end{equation}

Let $f,$ 
 $f\in \mathcal{O}(\mathbb{D}_{R}),$ then:
$$ d_{\mathcal{R}(p,q)}=(dz)\partial_{\mathcal{R}( p,q)} f(z).$$
The following identities are true:
\begin{itemize}
	\item [(a)]$d_{\mathcal{R}(p,q)}(fg)=(d_{z})\frac{p-q}{p^{P}-q^{Q}}\mathcal{R}(p^{P},q^{Q})\partial_{p,q}\left\{\partial_{p,q}(f)(Pg)+(Qf)(\partial_{p,q}(g))  \right\},$ 
	
	\item [(b)]$d_{\mathcal{R}(p,q)}(fg)=(d_{z})\frac{p-q}{p^{P}-q^{Q}}\mathcal{R}(p^{P},q^{Q})\partial_{p,q}\left\{\partial_{p,q}(f)(Qg)+(Pf)(\partial_{p,q}(g))  \right\},$ 
	
\end{itemize}
for $f,g\in \mathcal{O}(\mathbb{D}_{R}).$	
\\
We define the operator $I_{\mathcal{R}( p,q)}$ over $\mathcal{O}(\mathbb{D}_{R})$ as the inverse image of the \\
$\mathcal{R}( p,q)-$derivative. For the elements $z^{n}$ of the basis of $\mathcal{O}(\mathbb{D}_{R})$, $I_{\mathcal{R}( p,q)}$ acts as follows:
\begin{align}
	I_{\mathcal{R}( p,q)}z^{n}:=(\partial_{\mathcal{R}( p,q)})^{-1}z^{n}=\dfrac{1}{[N+1]_{\mathcal{R}( p,q)}}z^{n+1}+c,
\end{align}
where $n\geq 0$ and $c$ is an integration constant. 
If $f\in\mathcal{O}(\mathbb{D}_{R}),$ then 

\begin{align*}
	I_{\mathcal{R}( p,q)}\partial_{\mathcal{R}( p,q)}f(z)=f(z)+c,
	\\
	\partial_{\mathcal{R}( p,q)}I_{\mathcal{R}( p,q)}f(z)=f(z)+c^{\prime},
\end{align*}
where $c$ and $c^{\prime}$ are integrated constants. Provided $\mathcal{R}(P,Q)$ is invertible, one can define the $\mathcal{R}(p,q)-$integration by the following formula:

\begin{align}
	I_{\mathcal{R}( p,q)}= R^{-1}(P,Q)z.
\end{align}

We can also derive the definite Integrals:
\begin{align*}
	\int_{\alpha}^{\beta}f(z)d_{\mathcal{R}(p,q)}z=I_{\mathcal{R}( p,q)}f(\beta)- I_{\mathcal{R}( p,q)}f(\alpha)\\
	\int_{\alpha}^{+\infty}f(z)d_{\mathcal{R}(p,q)}z=\lim_{n\longrightarrow\infty}\int_{\alpha}^{\dfrac{p^{n}}{q^{n}}}f(z)d_{\mathcal{R}(p,q)}
	\\
	\int_{-\infty}^{+\infty}f(z)d_{\mathcal{R}(p,q)}z=\lim_{n\longrightarrow\infty}\int_{-\dfrac{p^{n}}{q^{n}}}^{\dfrac{p^{n}}{q^{n}}}f(z)d_{\mathcal{R}(p,q)}
\end{align*}
Integrating  by parts:
\begin{align*}
	I_{\mathcal{R}( p,q)}\partial_{\mathcal{R}( p,q)}(f(z)g(z))&=f(z)g(z)+c	\\
	=	& I_{\mathcal{R}( p,q)} \left\lbrace \frac{p-q}{p^{P}-q^{Q}}\mathcal{R}(p^{P},q^{Q})\partial_{p,q}\left\{\partial_{p,q}(f)(Pg)  \right\}\right\rbrace \\
	&\left. +I_{\mathcal{R}( p,q)} \left\lbrace\frac{p-q}{p^{P}-q^{Q}}\mathcal{R}(p^{P},q^{Q})(Qf)(\partial_{p,q}(g)) \right\rbrace,\right. \\
	I_{\mathcal{R}( p,q)}\partial_{\mathcal{R}( p,q)}(f(z)g(z))&=f(z)g(z)+c\\
	=&	I_{\mathcal{R}( p,q)} \left\lbrace \frac{p-q}{p^{P}-q^{Q}}\mathcal{R}(p^{P},q^{Q})\partial_{p,q}\left\{\partial_{p,q}(f)(Qg)  \right\}\right\rbrace\\
	&\left. +I_{\mathcal{R}(p,q)} \left\lbrace\frac{p-q}{p^{P}-q^{Q}}\mathcal{R}(p^{P},q^{Q})(Pf)(\partial_{p,q}(g)) \right\rbrace. \right.\\
\end{align*}

\subsection{ $\mathcal{R}(p,q)-$analogue of Sadjang $(p,q)-$integrals}
Provided the above results, we consider a formal power series $$f(z)=\sum_{n=0}^{\infty}a_{n}z^{n}.$$ Then, its anti-derivative will be given by
\begin{equation}
	\int f(z)d_{\mathcal{R}(p,q)}z=\sum_{n=0}^{\infty}\dfrac{a_{n}z^{n+1}}{[n+1]_{\mathcal{R}( p,q)}}+C.
\end{equation}
Let $f(z)$ and $F(x)$ be arbitrary functions such that
\begin{align*}
	d_{\mathcal{R}(p,q)}F(z)&=\partial_{\mathcal{R}( p,q)}F(z) \frac{p-q}{P-Q}\mathcal{R}(P,Q) \\
	&=\frac{p-q}{p^{P}-q^{Q}}\mathcal{R}(p^{P},q^{Q}) \partial_{( p,q)}F(z)\\
	&=\frac{p-q}{p^{P}-q^{Q}}\mathcal{R}(p^{P},q^{Q})\dfrac{F(pz)-F(qz)}{(p-q)z}\\
	&=\frac{p-q}{p^{P}-q^{Q}}\mathcal{R}(p^{P},q^{Q})f(z).
\end{align*}

Thus, 
\begin{align}
	\frac{p-q}{p^{P}-q^{Q}}\mathcal{R}(p^{P},q^{Q})	\dfrac{F(pz)-F(qz)}{(p-q)z}=f(z), 
\end{align}
which can be rewritten  as:
\begin{align*}
	\frac{p-q}{p^{P}-q^{Q}}\mathcal{R}(p^{P},q^{Q})	F(pz)-F(qz)=(p-q)zf(z). 
\end{align*}
Setting $(p-q)=\eta,$ then,
\begin{align*}
	\frac{p-q}{p^{P}-q^{Q}}\mathcal{R}(p^{P},q^{Q})	F(pz)-F(qz)=\eta zf(z).
\end{align*}
By induction, one can obtain the following:
\begin{align*}
	\frac{p^{1}-q^{1}}{p^{P}-q^{Q}}\mathcal{R}(p^{P},q^{Q})	F(pq^{-1}z)-F(p^{0}q^{0}z)&=\eta p^{0}q^{-1} zf(p^{0}q^{-1} z),
	\\
	\frac{p^{2}-q^{2}}{p^{P}-q^{Q}}\mathcal{R}(p^{P},q^{Q})	F(p^{2}q^{-2}z)-F(p^{1}q^{-1}z)&=\eta p^{1}q^{-2} zf(p^{1}q^{-2} z),
	\\
	&\vdots
	\\
	\frac{p^{n}-q^{n}}{p^{P}-q^{Q}}\mathcal{R}(p^{P},q^{Q})	F(p^{n+1}q^{-(n+1)}z)-F(p^{n}q^{-n}z)&=\eta p^{n}q^{-(n+1)} z \\
	&\times f(p^{n}q^{-(n+1)} z).
\end{align*}
It is easy to see that:
$$F(p^{n+1}q^{-(n+1)}z)-F(z)=\frac{p^{P}-q^{Q}}{p^{n}-q^{n}}\mathcal{R^{-\prime}}(p^{P},q^{Q}) (p-q)z\sum_{r=0}^{n} f(p^{r}q^{-(r+1)} z).$$
Suppose $\mid p \mid < \mid q \mid$ as $n\longrightarrow \infty.$ Then,
\begin{align*}
	F(x)-F(0)= \frac{p^{P}-q^{Q}}{p^{n}-q^{n}}\mathcal{R^{-\prime}}(p^{P},q^{Q})(q-p)z\sum_{r=0}^{n}p^{r}/ q^{r+1} f(p^{r}/q^{(r+1)} z).
\end{align*}
Similarly, for $\mid q \mid < \mid p \mid$, we have 

\begin{align*}\label{3.2}
	F(x)-F(0)= \frac{p^{P}-q^{Q}}{p^{n}-q^{n}}\mathcal{R^{-\prime}}(p^{P},q^{Q})(p-q)z\sum_{r=0}^{n}q^{r}/ p^{r+1} f(q^{r}/p^{(r+1)} z).
\end{align*}
This leads to the following:
\begin{thm}
	\begin{align}
		\int f(z)d_{\mathcal{R}(p,q)}z=\frac{p^{P}-q^{Q}}{p^{n}-q^{n}}\dfrac{(p-q)z}{\mathcal{R}(p^{P},q^{Q})}\sum_{r=0}^{n}q^{r}/ p^{r+1} f(q^{r}/p^{(r+1)} z).
	\end{align}
\end{thm}
\bigskip
For two positive real numbers $p,q$ such that $ 0<q<p\leq 1$, we can rewrite 
\begin{align*}
	0<\dfrac{q}{p}<1\leq \dfrac{1}{p}\\
	0<\dfrac{q}{p}<1\leq \phi_{1}
\end{align*}
or, similarly, 
\begin{align*}
	0<\dfrac{p}{q}<1\leq \dfrac{1}{q}\\
	0<\dfrac{p}{q}<1\leq \phi_{2}.
\end{align*}

\begin{thm}\label{def 3.4}
	\begin{align}
		\int_{0}^{a} f(z)d_{\mathcal{R}(p,q)}z=\frac{q^{Q}-p^{P}}{q^{n}-p^{n}}\dfrac{(q-p)a}{\mathcal{R}(q^{Q},p^{P})}\sum_{r=0}^{n}p^{r}/q^{r+1} f(p^{r}/q^{(r+1)} a)
	\end{align}
	if $\bigg| \dfrac{p}{q}\bigg|<1\leq \phi_{1}$.
	\begin{align}\label{def 3.27}
		\int_{0}^{a} f(z)d_{\mathcal{R}(p,q)}z=\frac{p^{P}-q^{Q}}{p^{n}-q^{n}}\dfrac{(p-q)a}{\mathcal{R}(p^{P},q^{Q})}\sum_{r=0}^{n}q^{r}/ p^{r+1} f(q^{r}/p^{(r+1)} a)
	\end{align}
	if $\bigg| \dfrac{q}{p}\bigg|<1\leq \phi_{2}$.
\end{thm}
\begin{rem}\label{rem3.5}
	Suppose $|f(x)z^{\gamma}|<M$ on $(0,A]$. For any $0<z<A, \quad j\geq 0$, we have 
	$$\bigg|f\left(\dfrac{q^{j}}{p^{j+1}}z\right)\bigg|<M \left(\dfrac{q^{j}}{p^{j+1}}z\right)^{-\gamma}.$$
	For $0<z\leq A$, we can observe that
	$$\bigg|\dfrac{q^{j}}{p^{j+1}}f\left(\dfrac{q^{j}}{p^{j+1}}z\right)\bigg|<M \dfrac{q^{j}}{p^{j+1}} \left(\dfrac{q^{j}}{p^{j+1}}z\right)^{-\gamma}=Mp^{\gamma -1}z^{-\gamma}\left[\left(\dfrac{q}{p}\right)^{1-\gamma} \right]^{j}.$$
\end{rem}
Since, $1-\gamma>0$ and $0<\dfrac{q}{p}<1\leq \phi_{1},$ it is easy to see that the integral converges. Furthermore,
\begin{align}
	\bigg|\frac{p^{P}-q^{Q}}{p^{n}-q^{n}}\dfrac{(p-q)z}{\mathcal{R}(p^{P},q^{Q})}\sum_{r=0}^{\infty}\dfrac{q^{j}}{p^{j+1}}f\left(\dfrac{q^{j}}{p^{j+1}}z\right)\bigg|&<M\dfrac{(p-q)z^{1-\gamma}}{p^{1-\gamma}-q^{1-\gamma}}\\
	\nonumber
	&\times\frac{p^{P}-q^{Q}}{p^{n}-q^{n}}\mathcal{R^{-\prime}}(p^{P},q^{Q}),
\end{align}
for $\phi_{1}\geq 1$ and $0<z\leq A.$
\\
Next, let $F(z)$ be a continuous function at $z=0$, then from equation $\eqref{3.2}$  we have 

\begin{align*}
	F(z)=\frac{p^{P}-q^{Q}}{p^{n}-q^{n}}\dfrac{(p-q)z}{\mathcal{R}(p^{P},q^{Q})}\sum_{r=0}^{n}q^{r}/ p^{r+1} f(q^{r}/p^{(r+1)} z)-F(0).
\end{align*}
This leads to 
\begin{align*}
	\int_{0}^{a} f(z)d_{\mathcal{R}(p,q)}z=F(a)- F(0),
\end{align*}
and similarly, 
\begin{align*}
	\int_{0}^{b} f(z)d_{\mathcal{R}(p,q)}z=F(b)- F(0)
\end{align*}
from Theorem 
\ref{def 3.4},
 where $a$ and $b$ are finite.
Finally, we have

\begin{align*}
	\int_{a}^{b} f(z)d_{\mathcal{R}(p,q)}z=\int_{0}^{b} f(z)d_{\mathcal{R}(p,q)}z- \int_{0}^{a} f(z)d_{\mathcal{R}(p,q)}z
\end{align*}
where $a<b$. 
From Theorem  
\ref{def 3.4},
 the equation \eqref{def 3.27} yields a bad definition of improper integral by simply as $a\longrightarrow \infty$. 
\begin{align*}
	\int_{\frac{q^{j+1}}{p^{j+1}}}^{\frac{q^{j}}{p^{j}}} f(z)d_{\mathcal{R}(p,q)}z&=\int_{0}^{\frac{q^{j}}{p^{j}}} f(z)d_{\mathcal{R}(p,q)}z- \int_{0}^{\frac{q^{j+1}}{p^{j+1}}} f(z)d_{\mathcal{R}(p,q)}z\\
	&=\frac{p^{P}-q^{Q}}{p^{n}-q^{n}}\dfrac{(p-q)}{\mathcal{R}(p^{P},q^{Q})}\left[ \sum_{r=0}^{n}q^{r+j}/ p^{r+j+1} f(q^{r+j}/p^{(r+j+1)}\right] \\
	&-\frac{p^{P}-q^{Q}}{p^{n}-q^{n}}\dfrac{(p-q)}{\mathcal{R}(p^{P},q^{Q})}\left[\sum_{r=0}^{n}q^{r+j+1}/ p^{r+j+2} \right.\\&\left.
	\times f(q^{r+j+1}/p^{(r+j+2)})\right] \\
	&=\frac{p^{P}-q^{Q}}{p^{n}-q^{n}}\dfrac{(p-q)}{\mathcal{R}(p^{P},q^{Q})}q^{j}/ p^{j+1} f(q^{j}/p^{(j+1)}),
\end{align*}
 leading  to the following:

\begin{prop}
	Suppose $0<\frac{q}{p}<1\leq\phi_{1}$, then for $f(z)$ on $[0,\infty)$, the improper $\mathcal{R}(p,q)-$ integral is given as
	\begin{align*}
		\int_{0}^{\infty} f(z)d_{\mathcal{R}(p,q)}z&=\sum_{j=-\infty}^{\infty}\int_{\frac{q^{j+1}}{p^{j+1}}}^{\frac{q^{j}}{p^{j}}} f(z)d_{\mathcal{R}(p,q)}z
		\\
		&=\frac{p^{P}-q^{Q}}{p^{n}-q^{n}}\dfrac{(p-q)}{\mathcal{R}(p^{P},q^{Q})}\sum_{j=-\infty}^{\infty}q^{j}/ p^{j+1} f(q^{j}/p^{(j+1)}).
	\end{align*}
	Similarly, for $0<\frac{q}{q}<1\leq\phi_{2}$, we have: 
	\begin{align*}
		\int_{0}^{\infty} f(z)d_{\mathcal{R}(p,q)}z&=\sum_{j=-\infty}^{\infty}\int_{\frac{q^{j}}{p^{j}}}^{\frac{q^{j+1}}{p^{j+1}}} f(z)d_{\mathcal{R}(p,q)}z.
	\end{align*}
\end{prop}

\bigskip
By Remark
 \ref{rem3.5}, 
 we have established the convergence of the first part: 
\begin{align*}
	\int_{0}^{\infty} f(z)d_{\mathcal{R}(p,q)}z&=\frac{p^{P}-q^{Q}}{p^{n}-q^{n}}\dfrac{(p-q)}{\mathcal{R}(p^{P},q^{Q})}\sum_{j=-\infty}^{\infty}q^{j}/ p^{j+1} f(q^{j}/p^{(j+1)})\\
	&=\frac{p^{P}-q^{Q}}{p^{n}-q^{n}}\dfrac{(p-q)}{\mathcal{R}(p^{P},q^{Q})}\left[\sum_{j=-\infty}^{\infty}q^{j}/ p^{j+1} f(q^{j}/p^{(j+1)})\right]\\
	&+\frac{p^{P}-q^{Q}}{p^{n}-q^{n}}\dfrac{(p-q)}{\mathcal{R}(p^{P},q^{Q})}\left[\sum_{j=-\infty}^{\infty}q^{j}/ p^{j+1} f(q^{j}/p^{(j+1)})\right].
\end{align*}
Now, for the final part, we asume for every large $z$ that we have $|f(x)z^{\gamma}|<M$ with $M>0$, $\gamma >0$ and $\delta\leq 1.$ 
For sufficiently large $j,$ we obtain
\begin{align*}
	\bigg|\dfrac{q^{-j}}{p^{-j+1}}f\left(\dfrac{q^{-j}}{p^{-j+1}}z\right)\bigg|&=p^{\gamma -1}\left(\dfrac{q}{p}\right)^{(\gamma-1)j} \bigg|\dfrac{q^{-j}}{p^{-j+1}} f\left (\dfrac{q^{-j}}{p^{-j+1}}z\right)\bigg|\\
	&<Mp^{\gamma -1}\left[\left(\dfrac{q}{p}\right)^{1-\gamma} \right]^{-j}.
\end{align*}
The series also converges. 

\begin{defn}
	Let $f^{\prime}(z)$ denote the ordinary derivative of $f(z).$ For a continuous $f^{\prime}(z)$, there exists a neighborhood $z=0$ such that $$ \int_{a}^{b}\partial_{\mathcal{R}( p,q)}f(z)d_{\mathcal{R}( p,q)}z=f(b)-f(a).$$
	For two functions whose the ordinary derivatives  exist in a neighbourhood of $z=0$,  the $\mathcal{R}( p,q)-$integration by parts is given as
	
	\begin{align*}
		\int_{a}^{b}f(px)[\partial_{\mathcal{R}( p,q)}g(z)]d_{\mathcal{R}( p,q)}z &=f(b)g(b)- f(a)g(a)\\
		&+\int_{a}^{b}g(qx)[\partial_{\mathcal{R}( p,q)}f(z)]d_{\mathcal{R}( p,q)}z,
	\end{align*}
	where $a,b$ are integers such that $0\leq a \leq b \leq \infty.$
\end{defn}

\begin{rem}
	By setting $\mathcal{R}(p,q)=1$,  the
	$(p,q)-$Jagannathan-Srinivasa derivative $f$ is defined as
	\begin{align}
		d_{p,q}f(z)=\dfrac{f(pz)-f(qz)}{(p-q)z}, \quad \quad z\neq 0,
	\end{align}	
	and $(d_{p,q}f)(0)=f^{\prime}(0)$, provided $f$ is differentiable at $0$. Thus the \\$(p,q)-$number is defined as
	
	\begin{eqnarray}
		[n]_{p,q}:&=& p^{n-1}+p^{n-2}q+p^{n-3}q^{2}+\cdots
		+pq^{n-2}+q^{n-1}\\ \nonumber
		&=&\dfrac{p^{n}-q^{n}}{p-q}.
	\end{eqnarray}
	It is clear that $d_{p,q}z^{n}=[n]_{p,q}z^{n-1}$. Consequently, for $p=1,$ the $(p,q)-$ derivative reduces to the $q-$derivative given by: 
	\begin{align}
		d_{q}f(z)=\dfrac{f(z)-f(qz)}{(1-q)z}, \quad \quad z\neq 0.
	\end{align}	
\end{rem}

\begin{enumerate}
	\item [(i)] $d_{p,q}(af(z)+bg(z))=a[d_{p,q}f(x)]+b[d_{p,q}g(z)], \quad \mbox{where}\quad a,b \quad \mbox{constants},$\\
	\item [(ii)]$ d_{p,q}(f(z)g(z))=f(px)d_{p,q}g(z)+g(qz)d_{p,q}f(z),$\\
	\item [(iii)]$d_{p,q}(f(z)g(z))=g(pz)d_{p,q}f(z)+f(qz)d_{p,q}g(z),$\\
	\item [(iv)]$d_{p,q}\left( \dfrac{f(z)}{g(z)}\right)= \dfrac{g(pz)d_{p,q}f(z)-f(pz)d_{p,q}g(z)}{g(pz)g(qz)},$\\
	\item [(iv)]$d_{p,q}\left( \dfrac{f(x)}{g(x)}\right)= \dfrac{g(qz)d_{p,q}f(z)-f(qz)d_{p,q}g(z)}{g(pz)g(qz)}.$
	
\end{enumerate}


\section{$\mathcal{R}(p,q)-$gamma and beta functions}\label{Gamma}
Sadjang \cite{ref58} offered two acceptable polynomial bases for the $(p, q)-$derivative, and different features of these bases were described. As an application, he provided two $(p, q)-$Taylor formulae for polynomials, and the fundamental theorem of $(p,q)-$calculus proves the $(p, q)-$integration by part formula. In addition, he introduced a novel generalization of the gamma and beta functions, known as the $(p, q)-$gamma and $(p, q)-$beta functions \cite{ref59}. Their primary characteristics were stated and demonstrated. It should be mentioned that the authors of \cite{ref80} proposed an additional generalization of the gamma function. In this section we extend the results to the $\mathcal{R}(p,q)-$deformed calculus and make some generalizations.

\subsection{$\mathcal{R}(p,q)-$gamma function}
\begin{defn}\label{def4.1}
	Let $z$ be a complex number. We define the $\mathcal{R}(p,q)-$gamma function as
	
	\begin{align}
		\Gamma_{\mathcal{R}(p,q)}(z)=\dfrac{(\xi_{1}\ominus \xi_{2})^{\infty}_{\mathcal{R}(p,q)}}{(\xi_{1}^{z}\ominus \xi_{2}^{z})^{\infty}_{\mathcal{R}(p,q)}}(\xi_{1}-\xi_{2})^{1-z}
	\end{align} 
	for $0<\xi_{2}<\xi_{1}$, with $\xi_{1}, \xi_{2}\in \xi_{i},\quad i=1,2$.
	Also, if $\xi_{1}=p \quad \mbox{and} \quad \xi_{2}=q$ with $\mathcal{R}(p,q)=1$ and $0<q<p$, one obtains the $\Gamma_{p,q}(z)$ function.This further reduces to $\Gamma_{q}(z)$ function, if we set $p=1$.
\end{defn}

The $\mathcal{R}(p,q)-$power basis and the $\mathcal{R}(p,q)-$ factorial are linked as
$$[n]_{\mathcal{R}(p,q)}!=\dfrac{(\xi_{1}\ominus \xi_{2})^{n}_{\mathcal{R}(p,q)}}{(\xi_{1}-\xi_{2})^{n}}.$$

\subsubsection{Properties of $\mathcal{R}(p,q)-$gamma function}
From the Definition \ref{def4.1} we have the following results: 
\begin{enumerate}
	\item [(i)] $\Gamma_{\mathcal{R}(p,q)}(z+1)=[z]_{\mathcal{R}(p,q)}\Gamma_{\mathcal{R}(p,q)}(z)$
	\item [(ii)]$\Gamma_{\mathcal{R}(p,q)}(n+1)=[n]_{\mathcal{R}(p,q)}!$ for a non-negative integer $n$ 	
	\item [(iii)] $\Gamma_{\mathcal{R}(p,q)}(2z)\Gamma_{\mathcal{R}(p^{2},q^{2})}\left(\dfrac{1}{2} \right)=(\xi_{1}+\xi_{2})^{2z-1}\Gamma_{\mathcal{R}(p^{2},q^{2})}(z)\Gamma_{\mathcal{R}(p^{2},q^{2})}(z+\frac{1}{2})$
	\item [(iv)] $\dfrac{1}{\Gamma_{\mathcal{R}(p,q)}(z)}=\prod_{n=0}^{\infty}\dfrac{(\xi_{1}\ominus \xi_{2}^{n+z})^{\infty}_{\mathcal{R}(p,q)}}{(\xi_{1}\ominus \xi_{2}^{n+1})^{\infty}_{\mathcal{R}(p,q)}}(\xi_{1}-\xi_{2})^{z-1}$	
\end{enumerate}

\subsection{$\mathcal{R}(p,q)-$beta function}
From the definition of the $\mathcal{R}(p,q)-$gamma function we  have the following:
\begin{align}
	\beta_{\mathcal{R}(p,q)}(x,y)&=\dfrac{\Gamma_{\mathcal{R}(p,q)}(x) \Gamma_{\mathcal{R}(p,q)}(y) }{\Gamma_{\mathcal{R}(p,q)}(x+y)},\\\nonumber
	&=\dfrac{\dfrac{(\xi_{1}\ominus \xi_{2})^{\infty}_{\mathcal{R}(p,q)}}{(\xi_{1}^{x}\ominus \xi_{2}^{x})^{\infty}_{\mathcal{R}(p,q)}}(\xi_{1}-\xi_{2})^{1-x}\dfrac{(\xi_{1}\ominus \xi_{2})^{\infty}_{\mathcal{R}(p,q)}}{(\xi_{1}^{y}\ominus \xi_{2}^{y})^{\infty}_{\mathcal{R}(p,q)}}(\xi_{1}-\xi_{2})^{1-y}}{\dfrac{(\xi_{1}\ominus \xi_{2})^{\infty}_{\mathcal{R}(p,q)}}{(\xi_{1}^{x+y}\ominus \xi_{2}^{x+y})^{\infty}_{\mathcal{R}(p,q)}}(\xi_{1}-\xi_{2})^{x-y}},
\end{align}
where $x,y \in \BZ_{p}$.

\subsubsection{Properties of the $\mathcal{R}(p,q)-$beta function}
\begin{enumerate}
	\item [(i)] $\beta_{\mathcal{R}(p,q)}(x,y+1)=\dfrac{[y]_{\mathcal{R}(p,q)}}{[x+y]_{\mathcal{R}(p,q)}}\beta_{\mathcal{R}(p,q)}(x,y)$
	\item [(ii)] $\beta_{\mathcal{R}(p,q)}(x+1,y)=\dfrac{[x]_{\mathcal{R}(p,q)}}{[x+y]_{\mathcal{R}(p,q)}}\beta_{\mathcal{R}(p,q)}(x,y)$
	\item [(iii)] $\beta_{\mathcal{R}(p,q)}(x+1,y)=\dfrac{[x]_{\mathcal{R}(p,q)}}{[y]_{\mathcal{R}(p,q)}}\beta_{\mathcal{R}(p,q)}(x,y+1)$
	\item [(iv)] $\beta_{\mathcal{R}(p,q)}(x+n,y)=\dfrac{(\xi_{1}^{x}\ominus \xi_{2}^{x})^{n}_{\mathcal{R}(p,q)}}{(\xi_{1}^{x+y}\oplus \xi_{2}^{x+y})^{n}_{\mathcal{R}(p,q)}}\beta_{\mathcal{R}(p,q)}(x,y)$
	\item [(v)] $\beta_{\mathcal{R}(p,q)}(x+1,y)+\beta_{\mathcal{R}(p,q)}(x,y+1)=\dfrac{[x]_{\mathcal{R}(p,q)}+ [y]_{\mathcal{R}(p,q)}}{[x+y]_{\mathcal{R}(p,q)}}\beta_{\mathcal{R}(p,q)}(x,y)$
	\item[(vi)] $\beta_{\mathcal{R}(p,q)}(x+1,y+1)=\dfrac{[x]_{\mathcal{R}(p,q)}+ [y]_{\mathcal{R}(p,q)}}{[x+y+1]_{\mathcal{R}(p,q)}[x+y]_{\mathcal{R}(p,q)}}\beta_{\mathcal{R}(p,q)}(x,y)$
	\item [(vii)] $\beta_{\mathcal{R}(p,q)}(x,y)+\beta_{\mathcal{R}(p,q)}(x+y,z)+\beta_{\mathcal{R}(p,q)}(x+y+z,w)\\ \nonumber
	=\dfrac{\Gamma_{\mathcal{R}(p,q)}(x) \Gamma_{\mathcal{R}(p,q)}(y)\Gamma_{\mathcal{R}(p,q)}(z)\Gamma_{\mathcal{R}(p,q)}(w) }{\Gamma_{\mathcal{R}(p,q)}(x+y+z+w)}$
	\item [(v)] $\beta_{\mathcal{R}(p,q)}(x,1-x)=\Gamma_{\mathcal{R}(p,q)}{(x)}\Gamma_{\mathcal{R}(p,q)}(1-x)=\dfrac{\pi}{\sin_{\mathcal{R}(p,q)}[\pi x]}$
	\item [(viii)] $\beta_{\mathcal{R}(p,q)}(x,y)\beta_{\mathcal{R}(p,q)}(x+y, 1-y)=\dfrac{\pi}{[x]_{\mathcal{R}(p,q)}\sin_{\mathcal{R}(p,q)}[\pi x]}$
	\item [(ix)] $\beta_{\mathcal{R}(p,q)}(\frac{1}{2}\frac{1}{2})=\pi.$
\end{enumerate}

\section{$\mathcal{R}(p,q)-$power basis and $\mathcal{R}( p,q)$-Taylor polynomial}
\begin{defn}
	Let $n$ be a non-negative integer. The $\mathcal{R}(p,q)-$analogues of {\it pseudo-addition and substraction} operations  are defined by
	\begin{enumerate}		
		\item [(i)] $(x\ominus a)^{n}_{\mathcal{R}(p,q)}=(x-a)(x\xi_{1}-a\xi_{2})(x\xi_{1}^{2}-a \xi_{2}^{2})\cdots(x\xi_{1}^{n-1}-a\xi_{2}^{n-1})$
		\item [(ii)] $(x\ominus a)^{-n}_{\mathcal{R}(p,q)}= \dfrac{1}{(x\xi_{1}^{-n}\ominus \xi_{2}^{-n}a)^{n}_{\mathcal{R}(p,q)}}$
		\item [(ii)] $(x\ominus a)^{-n}_{\mathcal{R}(p,q)}=(x\ominus a)^{m}_{\mathcal{R}(p,q)} (x\xi_{1}^{m}\ominus \xi_{2}^{m}a)^{n}_{\mathcal{R}(p,q)}$
		\item [(iii)] $(x\oplus a)^{n}_{\mathcal{R}(p,q)}=(x+a)(x\xi_{1}+\xi_{2}a)(x\xi_{1}^{2}+a\xi_{2}^{2})\cdots(x\xi_{1}^{n-1}+a\xi_{2}^{n-1})$\\\\
		if $n\geq 1$ respectively. We can further extend the results as follows:
		\item [(iv)] 	
		$(x\oplus y)^{\infty}_{\mathcal{R}(p,q)}=\prod_{i=0}^{\infty}(x\xi_{1}^{i} + y\xi_{2}^{i})$ 	
		\item [(v)] 	
		$(x\ominus y)^{\infty}_{\mathcal{R}(p,q)}=\prod_{i=0}^{\infty}(x\xi_{1}^{i} - y\xi_{2}^{i})$ 
		\item [(vi)]
		$(x\ominus y)^{n}_{\mathcal{R}(p,q)}=\dfrac{(x\ominus y)^{\infty}_{\mathcal{R}(p,q)}}{(x\xi_{1}^{n}\ominus y\xi_{2}^{n})^{\infty}_{\mathcal{R}(p,q)}}$
		\item [(vii)] 	
		$(x\xi_{1}^{n}\ominus y\xi_{2}^{n})^{k}_{\mathcal{R}(p,q)}=\dfrac{(x\ominus y)^{k}_{\mathcal{R}(p,q)}(x\xi_{1}^{k}\ominus y\xi_{2}^{k})^{n}_{\mathcal{R}(p,q)}}{(x\ominus y)^{n}_{\mathcal{R}(p,q)}}$
		\item [(viii)] $(x\ominus y)^{n+k}_{\mathcal{R}(p,q)}=(x\ominus y)^{n}_{\mathcal{R}(p,q)}(x\xi_{1}^{n}\ominus y\xi_{2}^{n})^{k}_{\mathcal{R}(p,q)} $
		\item [(ix)] 	
		$(x\xi_{1}^{k}\ominus y\xi_{2}^{k})^{n-k}_{\mathcal{R}(p,q)}=\dfrac{(x\ominus y)^{n}_{\mathcal{R}(p,q)}}{(x\ominus y)^{k}_{\mathcal{R}(p,q)}}$
		\item [(x)] 	
		$(x\xi_{1}^{2k}\ominus y\xi_{2}^{2k})^{n-k}_{\mathcal{R}(p,q)}=\dfrac{(x\ominus y)^{n}_{\mathcal{R}(p,q)}(x\xi_{1}^{n}\ominus y\xi_{2}^{n})^{k}_{\mathcal{R}(p,q)}}{(x\ominus y)^{2k}_{\mathcal{R}(p,q)}}$
		\item [(xi)] 	
		$(x\ominus y)^{2n}_{\mathcal{R}(p,q)} = (x\ominus y)^{n}_{\mathcal{R}(p^{2},q^{2})} (x \xi_{1}\ominus y\xi_{2})^{n}_{\mathcal{R}(p^{2},q^{2})}$ 
		\item [(xii)] 	
		$(x\ominus y)^{3n}_{\mathcal{R}(p,q)} = (x\ominus y)^{n}_{\mathcal{R}(p^{2},q^{2})}(x\xi_{1}\ominus y\xi_{2})^{n}_{\mathcal{R}(p^{3},q^{3})}\\
		\times (x \xi_{1}^{2}\ominus y\xi_{2}^{2})^{n}_{\mathcal{R}(p^{3},q^{3})} $ 
		\item [(xiii)] 	
		$(x\ominus y)^{3n}_{\mathcal{R}(p,q)} = (x\ominus y)^{n}_{\mathcal{R}(p^{2},q^{2})}(x\xi_{1}\ominus y\xi_{2})^{n}_{\mathcal{R}(p^{3},q^{3})}\\
		\times (x \xi_{1}^{2}\ominus y\xi_{2}^{2})^{n}_{\mathcal{R}(p^{3},q^{3})} $ 	
		\item [(xiv)] 	
		$(x\oplus y)^{kn}_{\mathcal{R}(p,q)}=\prod_{i=0}^{k-1}(x\xi_{1}^{i} + y\xi_{2}^{i})^{n}_{\mathcal{R}(p^{k},q^{k})}$ 	
	\end{enumerate}
\end{defn}

\subsection{$\mathcal{R}( p,q)$- derivatives associated with the power basis}
\noindent The $\mathcal{R}( p,q)$- derivatives associated with the power basis satisfy the following:
\begin{align}\label{def 3.36}
	\partial_{\mathcal{R}( p,q)}(x\ominus a)^{n}_{\mathcal{R}( p,q)}&= [n]_{\mathcal{R}( p,q)}(\xi_{1}x \ominus a)^{n-1}_{\mathcal{R}( p,q)} \quad n\geq 1, \\
	\nonumber
	\partial_{\mathcal{R}( p,q)}^{k}(x\ominus a)^{n}_{\mathcal{R}( p,q)}&= \xi_{1}^{\left(\begin{array}{ccc}k\\2\end{array}\right)}\dfrac{[n]_{\mathcal{R}( p,q)}!}{[n-k]_{\mathcal{R}( p,q)}!}(\xi_{1}^{k}x \ominus a)^{n-k}_{\mathcal{R}( p,q)},	\\\nonumber
	\partial_{\mathcal{R}( p,q)}^{k}(x \ominus a)^{n}_{\mathcal{R}( p,q)}&= (-1)^{k}\xi_{2}^{\left(\begin{array}{ccc}k\\2\end{array}\right)}\dfrac{[n]_{\mathcal{R}( p,q)}!}{[n-k]_{\mathcal{R}( p,q)}!}(a \ominus \xi_{2}^{k}x)^{n-k}_{\mathcal{R}( p,q)},
	\\\nonumber
	\partial_{\mathcal{R}( p,q)}\dfrac{1}{(x \ominus a)^{n}_{\mathcal{R}( p,q)}}&=\dfrac{-\xi_{2}[n]_{\mathcal{R}( p,q)}}{(\xi_{2}x\ominus a)_{\mathcal{R}( p,q)}^{n+1}},
	\nonumber
	\\
	\partial_{\mathcal{R}( p,q)}\dfrac{1}{(a\ominus x)^{n}_{\mathcal{R}(p,q)}}&=\dfrac{\xi_{1}[n]_{\mathcal{R}( p,q)}}{(a \ominus\xi_{1}x)_{\mathcal{R}( p,q)}^{n+1}},
	\nonumber
	\\
	\partial_{\mathcal{R}( p,q)}(a\ominus x)^{n}_{\mathcal{R}( p,q)}&=-[n]_{\mathcal{R}( p,q)}(a \ominus \xi_{2}x)^{n-1}_{\mathcal{R}( p,q)}.
\end{align}

\begin{defn}
	For any given complex number $\lambda$ and a nonnegative integer $n$, the following relations hold:
	\begin{align}
		\partial_{\mathcal{R}( p,q)}e_{\mathcal{R}( p,q)}(\lambda x)&=\lambda e_{\mathcal{R}( p,q)}(\lambda \xi_{1}x),\\\nonumber
		\partial_{\mathcal{R}( p,q)}\mathbb{E}_{\mathcal{R}( p,q)}(\lambda x)&=\lambda \mathbb{E}_{\mathcal{R}( p,q)}(\lambda \xi_{1}x),\\
		\nonumber
		\partial_{\mathcal{R}( p,q)}^{n}e_{\mathcal{R}( p,q)}(\lambda x)&=\lambda^{n}\xi_{1}^{\left(\begin{array}{ccc}n\\2\end{array}\right)} e_{\mathcal{R}( p,q)}(\lambda \xi_{1}^{n}x),\\\nonumber
		\partial_{\mathcal{R}( p,q)}^{n}\mathbb{E}_{\mathcal{R}( p,q)}(\lambda x)&=\lambda^{n}\xi_{2}^{\left(\begin{array}{ccc}n\\2\end{array}\right)} \mathbb{E}_{\mathcal{R}( p,q)}(\lambda \xi_{2}^{n}x). \label{def4.3} 
	\end{align}
\end{defn}

\begin{thm}
	For a given complex number $a$, the following expansions hold:
	\begin{align}
		e_{\mathcal{R}( p,q)}(\lambda x)&=\sum_{n=0}^{\infty}\dfrac{((\xi_{1}- \xi_{2})\lambda)^{n}}{(\xi_{1}\ominus \xi_{2})^{n}_{\mathcal{R}( p,q)}}(x\ominus a)^{n}_{\mathcal{R}( p,q)},
		\\\nonumber
		\mathbb{E}_{\mathcal{R}( p,q)}(\lambda x)&=\sum_{n=0}^{\infty}\left(\dfrac{\xi_{2}}{\xi_{1}}\right)^{-\left(\begin{array}{ccc}n\\2\end{array}\right)}\dfrac{\lambda^{n}\mathbb{E}_{\mathcal{R}( p,q)}\left(\lambda a \left(\dfrac{\xi_{2}}{\xi_{1}}\right)^{n}\right)}{[n]_{\mathcal{R}( p,q)}!}(x\ominus a)^{n}_{\mathcal{R}( p,q)},
		\\
		\mathbb{E}_{p,q}(\lambda x)&=\mathbb{E}_{\mathcal{R}( p,q)}(\lambda a)\sum_{n=0}^{\infty}\dfrac{((\xi_{1}- \xi_{2})\lambda)^{n}}{(\xi_{1}\ominus \xi_{2})^{n}_{\mathcal{R}( p,q)}}(a\ominus x)^{n}_{\mathcal{R}( p,q)},\\\nonumber 
	\end{align}
	and
	\begin{align}
		e_{\mathcal{R}( p,q)}(\lambda x)&=\sum_{n=0}^{\infty}\left(-\dfrac{\xi_{1}}{\xi_{2}}\right)^{-\left(\begin{array}{ccc}n\\2\end{array}\right)}\dfrac{\lambda^{n} e_{\mathcal{R}( p,q)}\left(\lambda a \left(\dfrac{\xi_{1}}{\xi_{2}}\right)^{n}\right)}{[n]_{\mathcal{R}( p,q)}!}(a\ominus x)^{n}_{\mathcal{R}( p,q)}.
	\end{align}	
\end{thm}

\begin{proof}	
	Suppose we take $N$ to be $\infty$, then we can have the following:
	\begin{align*}
		f(x)=\sum_{k=0}^{\infty}\xi_{1}^{-\left(\begin{array}{ccc}k\\2\end{array}\right)}\dfrac{(\partial^{k}_{\mathcal{R}( p,q)}f)(a\xi_{1}^{-k})}{[k]_{\mathcal{R}( p,q)}!}(x\ominus a)^{k}_{\mathcal{R}( p,q)}.
	\end{align*}	
	If we take $f(x)=	e_{\mathcal{R}( p,q)}(\lambda x),$ one can use the relation $\eqref{def4.3}$ to obtain the following:
	\begin{align*}
		e_{\mathcal{R}( p,q)}(\lambda x)&=\sum_{n=0}^{\infty}\xi_{1}^{-\left(\begin{array}{ccc}n\\2\end{array}\right)}\dfrac{\lambda^{n}	e_{\mathcal{R}( p,q)}(\lambda a) \xi_{1}^{\left(\begin{array}{ccc}n\\2\end{array}\right)}}{[n]_{\mathcal{R}( p,q)}!}(x\ominus a)^{n}_{\mathcal{R}( p,q)}
		\\
		&=e_{\mathcal{R}( p,q)}(\lambda a)\sum_{n=0}^{\infty}\dfrac{\lambda^{n}}{[n]_{\mathcal{R}( p,q)}!}(x\ominus a)^{n}_{p,q}
		\\
		&=e_{\mathcal{R}( p,q)}(\lambda a)\sum_{n=0}^{\infty}\dfrac{((\xi_{1}- \xi_{2})\lambda)^{n}}{(\xi_{1}\ominus \xi_{2})^{n}_{\mathcal{R}( p,q)}}(x\ominus a)^{n}_{\mathcal{R}( p,q)}.
	\end{align*}
	Similarly, if we take $f(x)=\mathbb{E}_{\mathcal{R}( p,q)}(\lambda x),$ one can use the relation \eqref{def4.3} to obtain the following:
	\begin{align*}
		\mathbb{E}_{\mathcal{R}( p,q)}(\lambda x)&=\sum_{n=0}^{\infty}\xi_{1}^{-\left(\begin{array}{ccc}n\\2\end{array}\right)}\dfrac{\lambda^{n}	\mathbb{E}_{p,q}\left(\lambda a \left(\dfrac{\xi_{1}}{\xi_{2}}\right)^{n}\right) \xi_{2}^{\left(\begin{array}{ccc}n\\2\end{array}\right)}}{[n]_{\mathcal{R}( p,q)}!}(x\ominus a)^{n}_{\mathcal{R}( p,q)}
		\\
		&=\sum_{n=0}^{\infty}\left[\frac{\xi_{1}}{\xi_{2}}\right]^{\left(\begin{array}{ccc}n\\2\end{array}\right)}\dfrac{\lambda^{n}	\mathbb{E}_{\mathcal{R}( p,q)}\left(\lambda a \left(\dfrac{\xi_{1}}{\xi_{2}}\right)^{n}\right) }{[n]_{\mathcal{R}( p,q)}!}(x\ominus a)^{n}_{\mathcal{R}( p,q)}.
	\end{align*}
	
	Next, suppose we take $N$ to be $\infty$, then we  have the following:
	\begin{align*}
		f(x)=\sum_{n=0}^{\infty}(-1)^{n}\xi_{2}^{-\left(\begin{array}{ccc}n\\2\end{array}\right)}\dfrac{(\partial^{n}_{\mathcal{R}( p,q)}f)(a\xi_{2}^{-n})}{[n]_{\mathcal{R}( p,q)}!}(a\ominus x)^{n}_{\mathcal{R}( p,q)}.
	\end{align*}
	Assume $f(x)=e_{\mathcal{R}( p,q)}(\lambda x).$ One can use the relation $\eqref{def4.3}$ to obtain the following:
	\begin{align*}
		e_{\mathcal{R}( p,q)}(\lambda x)&=\sum_{n=0}^{\infty}(-1)^{n}\xi_{2}^{-\left(\begin{array}{ccc}n\\2\end{array}\right)}\dfrac{\lambda^{n}	e_{\mathcal{R}( p,q)}\left(\lambda a \left(\dfrac{\xi_{1}}{\xi_{2}}\right)^{n}\right) \xi_{1}^{\left(\begin{array}{ccc}n\\2\end{array}\right)}}{[n]_{\mathcal{R}( p,q)}!}\\
		\nonumber
		&\times(a\ominus x)^{n}_{\mathcal{R}( p,q)}
		\\
		&=\sum_{n=0}^{\infty} \left(-\dfrac{\xi_{1}}{\xi_{2}}\right)^{\left(\begin{array}{ccc}n\\2\end{array}\right)}\dfrac{\lambda^{n}e_{\mathcal{R}( p,q)} \left(\lambda a \left(\dfrac{\xi_{1}}{\xi_{2}}\right)^{n}\right) }{[n]_{\mathcal{R}( p,q)}!}(a \ominus x)^{n}_{\mathcal{R}( p,q)}.
	\end{align*}
	
	Similarly, for $f(x)=	\mathbb{E}_{\mathcal{R}( p,q)}(\lambda x),$ one can use the relation \eqref{def4.3} to get:
	\begin{align*}
		\mathbb{E}_{\mathcal{R}( p,q)}(\lambda x)&=\sum_{n=0}^{\infty}(-1)^{n}\xi_{2}^{-\left(\begin{array}{ccc}n\\2\end{array}\right)}\dfrac{\lambda^{n}	\mathbb{E}_{\mathcal{R}( p,q)}(\lambda a) \xi_{2}^{\left(\begin{array}{ccc}n\\2\end{array}\right)}}{[n]_{\mathcal{R}( p,q)}!}(a\ominus x)^{n}_{\mathcal{R}( p,q)}
		\\
		&=\mathbb{E}_{\mathcal{R}( p,q)}(\lambda a)\sum_{n=0}^{\infty}(-1)^{n}\dfrac{\lambda^{n}}{[n]_{\mathcal{R}( p,q)}!}(a\ominus x)^{n}_{\mathcal{R}( p,q)}
		\\
		&=\mathbb{E}_{\mathcal{R}( p,q)}(\lambda a)\sum_{n=0}^{\infty}\dfrac{((\xi_{2}- \xi_{1})\lambda)^{n}}{(\xi_{1}\ominus \xi_{2})^{n}_{\mathcal{R}( p,q)}}(a\ominus x)^{n}_{\mathcal{R}( p,q)}.
	\end{align*}
	
\end{proof}


\begin{thm}
	For any polynomial $f(x)$ of degree $N$, and any number $a$, we have  the following $\mathcal{R}( p,q)$-Taylor expansion:
	\begin{align}
		f(x)=\sum_{k=0}^{N}\xi_{1}^{-\left(\begin{array}{ccc}k\\2\end{array}\right)}\dfrac{(\partial^{k}_{\mathcal{R}( p,q)}f)(a\xi_{1}^{-k})}{[k]_{\mathcal{R}( p,q)}!}(x\ominus a)^{k}_{\mathcal{R}( p,q)}
	\end{align}
	and 
	\begin{align}
		f(x)=\sum_{k=0}^{N}(-1)^{k}\xi_{2}^{-\left(\begin{array}{ccc}k\\2\end{array}\right)}\dfrac{(\partial^{k}_{\mathcal{R}( p,q)}f)(a\xi_{2}^{-k})}{[k]_{\mathcal{R}( p,q)}!}(a\ominus x)_{\mathcal{R}( p,q)}^{k}.
	\end{align}
\end{thm}

\begin{proof}
	Given a polynomial $f$ with degree $N$, then we have the following expansion: 
	\begin{equation}\label{py3.44}
		f(x)=\sum_{j=0}^{N}c_{j}(x\ominus a)_{\mathcal{R}( p,q)}^{j}.
	\end{equation}	
	For an integer $k,$ one can apply $\partial_{\mathcal{R}( p,q)}^{k}$ on both sides of this polynomial
	for $0\leq k \leq N$, and from equation $\eqref{def 3.36}$ to obtain
	
	\begin{align}
		(\partial^{k}_{\mathcal{R}( p,q)}f)(x)=\sum_{j=k}^{N}c_{j}\dfrac{[j]_{\mathcal{R}( p,q)}!}{[j-k]_{\mathcal{R}( p,q)}!}\xi_{1}^{\left(\begin{array}{ccc}n\\2\end{array}\right)}(\xi_{1}^{k}x\ominus a)^{j-k}_{\mathcal{R}( p,q)}.
	\end{align}
	Setting $x=(a\xi_{1}^{-k})$ yields
	
	$$	(\partial^{k}_{\mathcal{R}( p,q)}f)(a\xi_{1}^{-k})=c_{k}[k]_{\mathcal{R}( p,q)}!\xi_{1}^{\left(\begin{array}{ccc}n\\2\end{array}\right)},$$
with
	\begin{align}
		c_{k}=\xi_{1}^{-\left(\begin{array}{ccc}n\\2\end{array}\right)} \dfrac{(\partial^{k}_{\mathcal{R}( p,q)}f)(a\xi_{1}^{-k})}{[k]_{\mathcal{R}( p,q)}}.
	\end{align}	
	Similarly, for any polynomial $f$ with degree $N$, then we have the following expansion: 
	\begin{equation}\label{py3.44}
		f(x)=\sum_{j=0}^{N}c_{j}(a\ominus x)_{\mathcal{R}( p,q)}^{j}.
	\end{equation}	
	For an integer $k,$ one can apply $\partial_{\mathcal{R}( p,q)}^{k}$ on both sides of the polynomial \eqref{py3.44} for $0\leq k \leq N$, and from equation \eqref{def 3.36},
	
	\begin{align}
		(\partial^{k}_{\mathcal{R}( p,q)}f)(x)=\sum_{j=k}^{N}c_{j}(-1)^{j}\dfrac{[j]_{\mathcal{R}( p,q)}!}{[j-k]_{\mathcal{R}( p,q)}!}\xi_{2}^{-\left(\begin{array}{ccc}k\\2\end{array}\right)}(a\ominus \xi_{2}^{k}x)^{j-k}_{\mathcal{R}( p,q)}.
	\end{align}
	Setting $x=(a\xi_{2}^{-k}),$ we obtain;
	
	$$	(\partial^{k}_{\mathcal{R}( p,q)}f)(a\xi_{2}^{-k})=c_{k}(-1)^{k}[k]_{\mathcal{R}( p,q)}!\xi_{2}^{-\left(\begin{array}{ccc}k\\2\end{array}\right)},$$
with
	\begin{align}
		c_{k}=(-1)^{k}\xi_{1}^{-\left(\begin{array}{ccc}k\\2\end{array}\right)} \dfrac{(\partial^{k}_{\mathcal{R}( p,q)}f)(a\xi_{2}^{-k})}{[k]_{\mathcal{R}( p,q)}!}.
	\end{align}	
	This ends the proof.	
\end{proof}

\section{$p$-adic spin Lie Group}\label{Adic}
Elementary particles have very interesting structures which link to many fields of mathematics. In \cite{ref33}, the $\Spin (\frac{1}{2})$(resp. Bosons) is a spin Lie group. These Lie groups have Lie algebras which can be complexified as follows 
\begin{gather*}
	\spin \left(\frac12,\mathbb{C} \right)=  \spin_{\BR} \left(\frac12\right) \oplus 
	i \spin_{\BR} \left(\frac12\right),
\end{gather*}
where $\spin_{\BR} \left(\frac12\right)$ is the real form of the spin half Lie algebra and for $\hbar=1$
$\spin_{\BR} \left(\frac12\right)\subset \sl(2,\BR)$. For the real Lie algebra of the spin half particle, the Iwasawa decomposition is given by 
\begin{equation*}
	\spin \left(\frac{1}{2}\right)= <S_{k}> \oplus  <S_{z}> \oplus <S_{+}>,
	\label{eq47}
\end{equation*} 
where $S_{k}$ is the set of skew symmetric $2\times 2$ matrices; $S_{z}$ is a set of $2\times 2$ real diagonal trace zero matrices; and $S_{+}$ is a set of upper triangular $2\times 2$ matrices with zeros on the diagonal.
The aim of this section is to answer questions $1, 2, 4$ and $5$ in the Introduction \ref{INT} and to show that the roots of the fermionic spin Lie group have a friendly ghost.
\subsubsection{Fermionic $p$-adic spin Lie group and its Lie algebra}
From our previous paper, the spin particle has Lie algebra structure with the integer spins (Bosons) and the odd integers half spins (fermions).
We have seen that the parastatistics have Hopf algebraic structure.
The Lie algebra $\spin(j)$ of spin particles can be represented by classical matrices, which makes it easier to see their algebraic nature \cite{ref15, ref17, ref22}: 
\[ 
\spin(j)=
\begin{cases}
	\text{higgs}    & \text{$j=0$;}\\
	\text{fermions} & \text{$j=(\frac{1}{2})\BZ_{p}$ when odd integer spins are considered};\\ 
	\text{bosons}   & \text{$j=\BZ$ when positive integer spins are considered}.
\end{cases}
\]	
 We denote by $p$ an odd-prime number and by $\mathbb{Q}_{p}$ the field of $p$-adic numbers with the normalized absolute value $\mid x\mid$, $x\in \mathbb{Q}_{p}$.
The Lie algebra $\sl(2n,\BC)$ can represent the fermion spin Lie algebra of elementary particles in quantum physics \cite{ref21}. As indicated in the mapping below. We define $j=(\frac{1}{2})\BZ_{p}$ as fraction of the form  \cite{ref6, ref21, ref22}:
\begin{equation*}
	\xymatrix{\sl(2n,\BZ_{p}) \ar[r] & \spin_{\BZ_{p}}\left(\frac{1}{2}\right) \ar[r] & \text{$p-$adic fermions}.}
\end{equation*}
The Lie group $\SL(2n,\BC)$ structure can represent the fermion spin Lie group analog \cite{ref6, ref21, ref22}: 	
\begin{equation*}
	\xymatrix{\SL(2n,\BZ_{p}) \ar[r] & \Spin_{\BZ_{p}}\left(\frac{1}{2}\right) \ar[r] & \text{$p-$adic fermions}.}
\end{equation*}

\begin{defn}
	A Poisson manifold is a smooth manifold $M$ together with a map
	$\left\lbrace ,\right\rbrace  :C^{\infty}(M) \times C^{\infty}(M) \longrightarrow C^{\infty}(M)$, the Poisson bracket, that satisfies:
	\begin{enumerate}
		\item [(i)]	antisymmetry: $ \left\lbrace f_{1},f_{2} \right\rbrace= -\left\lbrace f_{2},f_{1} \right\rbrace \mbox{for all} f_{1}, f_{2} \in C^{\infty}(M)$.
		\item [(ii)]	Leibniz identity: $\left\lbrace f_{1}\circ f_{2}, f_{3} \right\rbrace= \left\lbrace f_{1},f_{3} \right\rbrace f_{2} + \left\lbrace f_{2},f_{3} \right\rbrace f_{1}$.
		\item [(iii)]	Jacobi identity:  $\left\lbrace \left\lbrace f_{1}, f_{2}  \right\rbrace, f_{3} \right\rbrace + \left\lbrace \left\lbrace f_{2}, f_{3}  \right\rbrace, f_{1} \right\rbrace +\left\lbrace \left\lbrace f_{3}, f_{1}  \right\rbrace, f_{2} \right\rbrace=0$.
	\end{enumerate}
\end{defn}
A Hamiltonian manifold is a manifold together with a Hamiltonian structure \cite{ref30}.
A Hamiltonian map (Poisson map) from a Hamiltonian manifold $(M, \left\lbrace, \right\rbrace M)$ to a Hamiltonian manifold $(M_{1}, \left\lbrace, \right\rbrace M_{1})$ is a smooth map $\psi : M \longrightarrow M_{1}$ that satisfies
$$ \left\lbrace f_{1} \circ \psi, f_{2} \circ \psi \right\rbrace_{M}= \left\lbrace f_{1}, f_{2} \right\rbrace_{ M_{1}} \circ\psi$$	$ \mbox{for all}  f_{1}, f_{2}\in C^{\infty}(M_{1}).$

\subsubsection{Hamilton spin Lie group}
\begin{defn}\cite{ref30}
	Let the map $\mu:G \times G \longrightarrow G$ defined by $\mu(x, y)=xy$ be grouped, if it is Hamiltonian then $G$ is said to have a grouped Hamiltonian structure.	
\end{defn}

\begin{defn}	
	A spin Lie group $G$ together with a grouped Hamiltonian structure on it will be called a Hamilton-spin Lie group.
\end{defn}

\begin{ex}
	The elementary spin particles Lie groups such as Fermions spin Lie groups:
	$$\Spin \left( \frac12\right) , \Spin \left(\frac32 \right), \Spin \left(\frac52 \right), \Spin \left(\frac72 \right),\cdots $$ and Bosons $\Spin (1), \Spin (2)$ are all Hamilton-spin Lie groups.
\end{ex}


\begin{defn}
	The spin Lie group of all spin one-half particles with quantum state spanned by 2 states, $2\times2$ real matrices and determinant, $1$ when $\hbar=1$ is denoted by $\Spin_{\BR}(\frac12)$.
\end{defn}
\bigskip
The real Lie algebra $\g$ of the $\Spin_{\BR}(\frac12)$ is given by:
\begin{align}
	\spin_{\BR}\left(\frac{1}{2}\right)={\{S\in M_{2}(\BR) \mid \Tr S=0 }\}.
\end{align}


\begin{defn}
	A $p$-adic spin Lie group is a Hamilton-spin Lie group with a fermionic spin Lie group structure. 
\end{defn}
The $p$-adic spin Lie algebra of a spin half particle is given by:
\begin{align*}
	\spin_{\BZ_{p}}\left(\frac{1}{2}\right)\subset \spin_{\BR}\left(\frac{1}{2}\right) \subset \spin_{\BC}\left(\frac{1}{2}\right).
\end{align*}

Let $L$ be a finite extension of the $\mathbb{Q}_{p}$ contained in an algebraically closed field $k$, let $\varrho$ be the integer ring of $L$, and let $\rho$ be the maximal ideal of $\varrho$ \cite{ref54}.   
Let $G=\SL(2,L)$ and let $K=\SL(2,\varrho)$. Then $K$ is an open compact subgroup of $G$. We can observe that;
$\SL(2,L)\supset\Spin_{\BZ_{p}}\left(\frac12\right)$ and 
\begin{align}
	\g=\spin_{\BZ_{p}}\left(\frac{1}{2}\right)={\{S\in M_{2}(L) \mid \Tr S=0 }\}.
\end{align}
Then $\g$ becomes a Lie algebra with
\begin{align*}
	\mathfrak{g}\times \mathfrak{g}\ni(x,y)\mapsto [x,y]=xy-yx \in \mathfrak{g},
\end{align*} 
We consider $\spin_{\BZ_{p}}(\frac{1}{2})$ as the fermionic $p-$adic spin Lie algebra of the \\$ \Spin_{\BZ_{p}}\left(\frac12\right) \subset\SL(2,L).$ Since $K$ is an open subgroup of $\SL(2,L)$ and\\ $\spin_{\BZ_{p}}(\frac12)$ can also be regarded as the Lie algebra of $K$.


\begin{prop}\label{prop3}
	The fermionic $p$-adic spin Lie algebra, $\spin_{\BZ_{p}}(\frac{1}{2})$ can be generated by
	\begin{enumerate}
		\item the elements 
		$${S_{-}}=\hbar
		\begin{pmatrix}
			0 & 0\\
			1 & 0
		\end{pmatrix},\quad 
		S_{z}=\frac{\hbar}{2}
		\begin{pmatrix}
			1 & 0\\
			0 & -1
		\end{pmatrix},\quad 
		S_{+}=\hbar 
		\begin{pmatrix}
			0 & 1\\
			0 & 0
		\end{pmatrix}
		$$
		
		\item the commutation relations are given by:
		$$
		[{S_{z}},S_{+}]=2 \hbar S_{+},\quad 	
		[{S_{z}},S_{-}]=-2\hbar S_{-},\quad 
		[{S_{+}},S_{-}]=\hbar S_{z}.
		$$
	\end{enumerate} 
	
\end{prop}
\begin{proof} For $(1)$ 
	let $S$ be an element of $\spin_{\BZ_{p}}(\frac12)$. Since $k$ is an algebraically closed field containing $L$, there is $A\in \SL(2,k)$ such that $A^{-1}SA$ has the form  
	
	\begin{enumerate}
		\item [(i)]
		$$
		\begin{pmatrix}
			\hbar & 0\\
			0 & -\hbar
		\end{pmatrix}, \hbar \in k \quad
		\mbox{or}\quad
		(ii)	
		\begin{pmatrix}
			0 & 1\\
			0 & 0
		\end{pmatrix}.
		$$
	\end{enumerate} 
	Let $t$ be an element of $L$, and let 
	\begin{gather*}
		\exp(tS)=\sum_{n=0}^{\infty}\dfrac{(tS)^{n}}{n!} \quad \mbox{in }\quad   M_{2}(L).
	\end{gather*} 
	Since $$A^{-1}(tS)A=
	\begin{pmatrix}
		t\hbar & 0\\
		0 & -t\hbar
	\end{pmatrix} \quad \mbox{or} \quad \begin{pmatrix}
		0 & t\\
		0 & 0
	\end{pmatrix},$$
	this series converges for $|t\hbar|<|p^{1/(p-1)}|$ in case (i), and converges for any $t$ in case (ii). 
	If this condition is satisfied, $\exp(tS)$ is an element of $M_{2}(L)\cap\SL(2,k)=\SL(2,L)$. Since
	
	$$A^{-1}\exp(tS)A=
	\begin{pmatrix}
		\exp(t\hbar) & 0\\[0.5ex]
		0 & \exp(-t\hbar)
	\end{pmatrix}\quad \mbox{or} \quad \begin{pmatrix}
		1 & t\\
		0 & 1
	\end{pmatrix},$$\\
	$\exp(tS)$ satisfies $$|\Tr(\exp(tS))-2|=|(\exp(t\hbar)-1)(\exp(-t\hbar)-1)|<|p^{2/(p-1)}|$$
	and $|\Tr(\exp(tS))-2|=0$ in case (ii). We observe \cite{ref54} that any element $g\in\SL(2,L)$ satisfying $|\Tr(g)-2|<|p^{2/(p-1)}| $ can be written as $g=\exp(X)$ with
	\begin{gather*}
		X=\dfrac{\sum_{n=1}^{\infty}(g-1)^{n}(-1)^{n-1}}{n}=\spin_{\BZ_{p}}\left(\frac12\right).
	\end{gather*}
	Specifically, the image if the exponential map contains any sufficiently small principal congruence subgroups
	\begin{align}
		K_{n}=\left\{g\in SL(2, \varrho)| g=\begin{pmatrix}
			1 & 0\\
			0 & 1
		\end{pmatrix}\mod p^{n}  \right\}
	\end{align} 
	of $K$.
	The $(2)$ of this proposition is trivial.
\end{proof}
\begin{prop}
	For any $\spin_{\BZ_{p}} (\frac{1}{2})$ there exists elements $S_{+}, S_{-},  \mbox{and} \quad S_{z}$ which generate the Iwasawa algebras when we set $\hbar=1$.
\end{prop}
\begin{proof}
	the proof of this follows trivially from proposition \ref{prop3}.	
\end{proof}

\subsubsection{The principal congruence subgroup}
The principal congruence subgroup of $\Spin_{\BZ_{p}}\left(\frac{1}{2}\right)$  where $p$ is an odd prime.
We put 
\begin{align*}
	G_{i}&:= \left[\ker \Spin_{\BZ_{p}}\left(\frac{1}{2}\right) \rightarrow \Spin_{(\BZ_{p}/p^{i}\BZ)}\left(\frac{1}{2}\right)\right]
	\\
	\nonumber
	&= \left\lbrace I + A \mid \det (I +A)=1 \quad and \quad A\in M_{2}(p^{i}\BZ_{p})\right\rbrace .
\end{align*}
\begin{align*}
	L_{i}:= \spin_{(p^{i}\BZ_{p})}\left(\frac{1}{2}\right)= \left \lbrace A\in M_{2}(p^{i}\BZ_{p})| \Tr (A)=0 \right\rbrace.
\end{align*}
If $p$ is an odd prime then for all $i\leq1$, 
$$ G_{i}= \log (L_{i}); \quad L_{i}=\exp (G_{i})$$
it follows immediately that $G_{i+1}=G_{i}^{p}$ and $G_{i}$ is a uniform pro-$p$ group of dimension $3$.  
\\
Now we set $\hbar=1$ for the $\spin_{(p^{i}\BZ_{p})}\left(\frac{1}{2}\right)$ and choose a basis which contains a base for the Cartan  subalgebra, and a base for each of the root spaces such as 
\begin{enumerate}
	\item the elements 
	$${S_{-}}=p^{i}
	\begin{pmatrix}
		0 & 0\\
		1 & 0
	\end{pmatrix},\quad 
	S_{z}=\frac{p^{i}}{2}
	\begin{pmatrix}
		1 & 0\\
		0 & -1
	\end{pmatrix},\quad 
	S_{+}=p^{i} 
	\begin{pmatrix}
		0 & 1\\
		0 & 0
	\end{pmatrix}
	$$
	
	\item The commutation relations of such basis are:
	$$
	[{S_{z}},S_{+}]=2 p^{i} S_{+},\quad 	
	[{S_{z}},S_{-}]=-2 p^{i} S_{-},\quad 
	[{S_{+}},S_{-}]=p^{i} S_{z}
	$$
\end{enumerate} 
and the mapping 
$$(xS_{-}+yS_{z}+zS_{+}) \mapsto \left(\begin{array}{ccc}
	x\\
	y\\
	z
\end{array}\right)$$
is a bijection of $\spin_{(p^{i}\BZ_{p})}\left(\frac{1}{2}\right)$ and $(\BZ_{p})^{3}$.

\begin{thm}\label{5.8}\cite{ref4}
	Let $\mathcal{L}$ be a $3$-dimensional $\BZ_{p}$ - fermionic spin Lie algebra. Then there is a ternary quadratic form $f(x)\in \BZ_{p}[x_{1},x_{2},x_{3}]$ unique up to equivalence, such that, for $i\leq 0$,  $$\zeta_{p^{i},\mathcal{L}}(s)=\zeta_{\BZ_{p}^{3}}(s)-Z_{f}(s-2)\zeta_{p}(2s-2)p^{(2-s)(i+1)}(1-p^{-1})^{-1},$$
	where $Z_{f}(s)$ is Igusa's local zeta function associated to $f$.
\end{thm}


\subsubsection{Zeta function of fermion $\spin_{\BZ_{p}}\left(\frac{1}{2}\right)$ \mbox{Lie algebra}.} The $p$-adic spin Lie algebra of a spin half particle, $\spin_{\BZ_{p}}\left(\frac{1}{2}\right)$ has the direct sum:
\begin{align}
	\spin_{\BZ_{p}}\left(\frac{1}{2}\right)=\BZ_{p}S_{-}\oplus\BZ_{p}S_{+}\oplus\BZ_{p}S_{z},
\end{align}
where 
$$
[{S_{z}},S_{+}]=2 S_{+},\quad 	
[{S_{z}},S_{-}]=-2 S_{-},\quad 
[{S_{+}},S_{-}]= S_{z}.
$$
We obtain
$$R(x)= \left(\begin{array}{ccc}
	& x_{3} & -2x_{1}\\
	-x_{3} &  & 2x_{2}\\
	2x_{1}& -2x_{2} & 
\end{array}\right).$$
The ternary quadratic form is given by:
$f(x)=x_{3}^{2}+4x_{1}x_{2}$. Specifically for a fermion when $p=2 $ the  function $f$ defines a smooth conic in a projective 2-space which has a good reduction modulo $p$ and $p+1$ points over $\BF_{p}$.
Using the Igusa local zeta function we observe that;
$$Z_{f}(s-2)=\dfrac{(1-p^{-1})(1-p^{-1}t)}{(1-pt^{2})(1-pt)}.$$
From Theorem \ref{5.8} we obtain the formula
\begin{align}
	\zeta_{\spin_{\BZ_{p}}\left(\frac{1}{2}\right)}(s)&=\zeta_{\BZ_{p}^{3}}-\dfrac{(1-p^{-1}t)p^{2}t}{(1-pt)(1-p^{2}t)(1-pt^{2})(1-p^{2}t^{2})}\\
	\nonumber
	&=\zeta_{p}(s)\zeta_{p}(s-1)\zeta_{p}(2s-1)\zeta_{p}(2s-2)\zeta_{p}(3s-1)^{-1}.
\end{align}

\begin{thm}
	The only ghost polynomials associated with the fermionic $\Spin_{\BZ_{p}} (\frac{1}{2})$ Lie groups are the $\GO_{2l+1}$, $\GSp_{2l}$ or $\GO_{2l}^{+}$ of type $\Pi (B_{l})$, $\Pi (C_{l})$ and $\Pi( D_{l})$ respectively.
\end{thm}
\begin{proof}
	The proof follows from theorems; \ref{thm1}, \ref{thm4} \cite{ref33} and \ref{thmG} \cite{ref99}. 
\end{proof}

\begin{thm}\cite{ref99}
	Every fermionic ghost polynomial, $\GO_{2l+1}$, $\GSp_{2l}$ and $\GO_{2l}$ of type $\Pi (B_{l})$, $\Pi (C_{l})$ and $\Pi( D_{l})$ is a friendly ghosts.
\end{thm}



\section{Some $\mathcal{R}(p,q)-$polynomials and its applications}\label{Poly}
We will examine numerous polynomials in this section, along with their $\mathcal{R}(\rho,q)-$analogues, including the Volkenborn, Bernoulli, and Genocchi polynomials, before extending our extension to the $p-$adic integrals \cite{ref2}.
\subsection{$\mathcal{R}(\rho,q)-$Volkenborn}
Let $\rho, q\in\BC$ such that $0<|q|<|\rho|\leq 1$, or  ${p}$-adic numbers $\rho,q \in \BC_{p}$ such that 
\begin{eqnarray}
	|\rho-1 |_{{p}}<{p}^{-\frac{1}{{p}-1}} \quad \mbox{with}\quad \rho^{x}=\exp (x \log \rho), \\
	\nonumber
	|q-1 |_{{p}}<{p}^{-\frac{1}{{p}-1}} \quad \mbox{with}\quad q^{x}=\exp (x \log q),
\end{eqnarray}
the ${p}-$adic absolute value $|.|_{{p}}$, on $\BC_{{p}}$ normalized by $|{p}|_{{p}}=1 /{p}$, where $|x|_{{p}}\leq 1.$ 
For a locally constant function $f$, a ${p}-$adic distribution $\mu$ on $X$ is defined as $f: X \longrightarrow \mathbb{Q}_{{p}}.$
When $(a+ dp^{N}\BZ_{p})\subset \BZ_{p}$ extends to a ${p}-$adic distribution on $X$, then the ${p}-$adic distribution $\mu_{\mathcal{R}(\rho,q)}$ can be extended to a ${p}-$adic distribution on $X$.
Let $0<\bigg|\dfrac{q}{\rho}\bigg|<1$ and \\$D=\left\lbrace \dfrac{q}{\rho}\in\BC_{{p}}\bigg||q-\rho|<\rho p^{-\dfrac{1}{p-1}}  \right\rbrace$, $\bar{D}=\BC_{p}\setminus D$ be the complement of the open disc around $\rho$. \\ 
$\mu_{\mathcal{R}(\rho,q)}(a+ p^{N}\BZ_{p})$ is the measure if for $\dfrac{q}{\rho}\in \bar{D}, \quad \mbox{ord}_{p}(\rho-q)\neq -\infty$.

We say that $f$ is a uniformly differentiable function at a point $a\in \BZ$, which we put as $f\in UD(\BZ_{p})$, if the quotient
$$F_{f(x,y)}=\dfrac{f(x)-f(y)}{x-y}$$ have a limit as $f^{\prime}(a)$ as $(x,y)\rightarrow (a,a).$ 

\begin{thm}
	For $N \leq 1$ , ${p}-$adic distribution $\mu_{\mathcal{R}(\rho,q)}$ on $X$ is defined by 
	\begin{align}\label{6.1}
		\mu_{\mathcal{R}(\rho,q)}(a+ p^{N}\BZ_{p})= \dfrac{\rho^{{p}^{N}}(\rho^{\rho}-q^{Q})}{[{p}^{N}]_{\mathcal{R}(\rho,q)}}\dfrac{\mathcal{R^{-\prime}}(\rho^{\rho},q^{Q})}{(\rho^{n},q^{n})}\left(\dfrac{\rho}{q} \right)^{a}. 
	\end{align}
\end{thm} \label{4.1}
\begin{proof}
\begin{align*}
	\sum_{i=0}^{p-1}\mu_{\mathcal{R}(\rho,q)}(a+ip\BZ_{p}+ p^{N+1}\BZ_{p})&=\sum_{i=0}^{p-1}\dfrac{\rho^{{p}^{N+1}}}{[{p}^{N+1}]_{\mathcal{R}(\rho,q)}}\left(\dfrac{\rho}{q} \right)^{a+ ip^{N}}\\
	\nonumber
	&=\dfrac{(\rho-q)(\rho^{\rho}-q^{Q})}{\rho^{{p}^{N+1}}-q^{{p}^{N+1}}}\left(\dfrac{\rho}{q}\right)^{a}(\rho)^{{p}^{N+1}}\\
	\nonumber
	& \times \dfrac{\mathcal{R^{-\prime}}(\rho^{\rho},q^{Q})}{(\rho^{n},q^{n})}\sum_{i=0}^{p-1}\left(\dfrac{\rho}{q} \right)^{ip^{N}}\\
	\nonumber
	&=\dfrac{(\rho-q)(\rho^{\rho}-q^{Q})}{\rho^{{p}^{N+1}}-q^{{p}^{N+1}}}\left(\dfrac{\rho}{q}\right)^{a}(\rho)^{{p}^{N+1}}\\
	\nonumber
	&\times\dfrac{\mathcal{R^{-\prime}}(\rho^{\rho},q^{Q})}{(\rho^{n},q^{n})}\left(\dfrac{1-\left(\dfrac{\rho}{q}\right)^{p^{N+1}}}{1-\left(\dfrac{\rho}{q}\right)^{p^{N}}}\right)\\
	\nonumber
	&=\dfrac{(\rho-q)(\rho^{\rho}-q^{Q})}{\rho^{{p}^{N+1}}-q^{{p}^{N+1}}}\left(\dfrac{\rho}{q}\right)^{a}(\rho)^{{p}^{N+1}}\\
	&\times\dfrac{\mathcal{R^{-\prime}}(\rho^{\rho},q^{Q})}{(\rho^{n},q^{n})}\left(\dfrac{\rho^{{p}^{N+1}}-q^{{p}^{N+1}}}{\rho^{{p}^{N}}-q^{{p}^{N}}}\right)\dfrac{\rho^{{p}^{N}}}{\rho^{{P}^{N+1}}}
	\\
	\nonumber
	&=\dfrac{\rho^{{p}^{N}}(\rho^{\rho}-q^{Q})}{[{p}^{N}]_{\mathcal{R}(\rho,q)}}\dfrac{\mathcal{R^{-\prime}}(\rho^{\rho},q^{Q})}{(\rho^{n},q^{n})}\left(\dfrac{\rho}{q} \right)^{a}\\
	\nonumber
	&	=\mu_{\mathcal{R}(\rho,q)}(a+ {p}^{N}\BZ_{p}),
\end{align*}
this yields the desired Theorem \ref{4.1}.
\end{proof}

The $\mathcal{R}(\rho,q)-$ analogue Riemann sums for $f$ can be written as:
\begin{align}
	\sum_{0\leq a < p^{N}} f(a)	\mu_{\mathcal{R}(\rho,q)}(a+ p^{N}\BZ_{p})&=\dfrac{\rho^{{p}^{N}}(\rho^{\rho}-q^{Q})}{[{p}^{N}]_{\mathcal{R}(\rho,q)}}\dfrac{\mathcal{R^{-\prime}}(\rho^{\rho},q^{Q})}{(\rho^{n},q^{n})}\\
	\nonumber
	&\times \sum_{0\leq a < p^{N}}\left(\dfrac{\rho}{q} \right)^{a} f(a). 
\end{align}

The $\mathcal{R}(\rho,q)$-Volkenborn integral of a function $f\in UD(\BZ_{p})$ is defined by:
\begin{align}
	I_{\mathcal{R}(\rho,q)}(f)&=\int_{\BZ_{p}}f(x)d \mu_{\mathcal{R}(\rho,q)} (x)\\
	\nonumber
	&=\lim_{N\rightarrow \infty} \dfrac{\rho^{{p}^{N}}(\rho^{\rho}-q^{Q})}{[{p}^{N}]_{\mathcal{R}(\rho,q)}}\dfrac{\mathcal{R^{-\prime}}(\rho^{\rho},q^{Q})}{(\rho^{n},q^{n})}\sum_{x
		=0}\left(\dfrac{\rho}{q} \right)^{x} f(x)
\end{align}\label{6.8}
whose limit is convergent. We observe that $$\int_{\BZ_{p}} f_{n}(x)d\mu_{\mathcal{R}(\rho,q)}(x)\rightarrow \int_{\BZ_{p}} f(x)d\mu_{\mathcal{R}(\rho,q)}(x)$$
as $f_{n}\rightarrow f $ in $UD(\BZ_{p})$.

\begin{thm}
	For $f\in UD(\BZ_{p})$, we have 
	\begin{align}
		q^{n} I_{\mathcal{R}(\rho,q)}(f_{n})- \rho^{n} I_{\mathcal{R}(\rho,q)}&=\rho^{n}\dfrac{(\rho-q)}{\rho^{n}-q^{n}}\dfrac{(\rho^{\rho}-q^{Q})}{\mathcal{R}(\rho^{\rho},q^{Q})}\\
		\nonumber
		&\times \sum_{a=0}^{n-1}\left(\dfrac{q}{\rho} \right)^{a}\left(\dfrac{f^{\prime}(a)}{\ln q-\ln \rho}+f(a)\right)
	\end{align}
	where $f_{n}(x)=f(x+n)$ and for $n=1$ we obtain;
	\begin{align}
		q I_{\mathcal{R}(\rho,q)}(f_{1})- \rho I_{\mathcal{R}(\rho,q)}=\dfrac{\rho(\rho^{\rho}-q^{Q})}{\mathcal{R}(\rho^{\rho},q^{Q})}\left(\dfrac{f^{\prime}(0)}{\ln q-\ln \rho}+f(0)\right)
	\end{align} 
	where $f_{1}(x)=f(x+1)$.
	
\end{thm}
\begin{proof}
	Using Theorem $\ref{4.1}$ we have:
	
	\begin{align*}
		\left(\dfrac{q}{\rho}\right)^{n}I_{\mathcal{R}(\rho,q)} (f_{n})&=\lim_{N\rightarrow \infty} \dfrac{\rho^{{p}^{N}}(\rho^{\rho}-q^{Q})}{[{p}^{N}]_{\mathcal{R}(\rho,q)}}\dfrac{\mathcal{R^{-\prime}}(\rho^{\rho},q^{Q})}{(\rho^{n},q^{n})}\sum_{x=0}^{p^{N}-1}f(x+n)\left(\dfrac{q}{\rho} \right)^{x+n}\\
		&= \lim_{N\rightarrow \infty} \dfrac{\rho^{{p}^{N}}(\rho^{\rho}-q^{Q})}{[{p}^{N}]_{\mathcal{R}(\rho,q)}}\dfrac{\mathcal{R^{-\prime}}(\rho^{\rho},q^{Q})}{(\rho^{n},q^{n})}\left[ f(n)\left(\dfrac{q}{\rho} \right)^{n} \right.\\ &\left.
		+f(n+1)\left(\dfrac{q}{\rho} \right)^{n+1}+\cdots+  f(p^{N}+n-2)\left(\dfrac{q}{\rho} \right)^{p^{N}n-2}\right.\\ &\left. +f(p^{N}+n-1)\left(\dfrac{q}{\rho} \right)^{p^{N}n-1} \right]\\
		&= I_{\mathcal{R}(\rho,q)}(f)+ \lim_{N\rightarrow \infty} \dfrac{\rho^{{p}^{N}}(\rho^{\rho}-q^{Q})}{[{p}^{N}]_{\mathcal{R}(\rho,q)}}\dfrac{\mathcal{R^{-\prime}}(\rho^{\rho},q^{Q})}{(\rho^{n},q^{n})}\\
		\nonumber
		&\times \sum_{a=0}^{p^{N}-1}\left[ f(p^{N}+a)\left(\dfrac{q}{\rho} \right)^{p^{N}+a} 
		+f(a)\left(\dfrac{q}{\rho} \right)^{a}\right]\\
		&=I_{\mathcal{R}(\rho,q)}(f)+ \dfrac{(\rho-q)}{\rho^{n}-q^{n}}\dfrac{(\rho^{\rho}-q^{Q})}{\mathcal{R}(\rho^{\rho},q^{Q})} \sum_{a=0}^{n-1}\left(\dfrac{q}{\rho} \right)^{a}\\
		\nonumber
		&\times\dfrac{f^{\prime}(a)+f(a)\ln\frac{q}{\rho}}{\ln\frac{q}{\rho}}\\
		&=I_{\mathcal{R}(\rho,q)}(f)+ \dfrac{(\rho-q)}{\rho^{n}-q^{n}}\dfrac{(\rho^{\rho}-q^{Q})}{\mathcal{R}(\rho^{\rho},q^{Q})} \sum_{a=0}^{n-1}\left(\dfrac{q}{\rho} \right)^{a}\\
		\nonumber
		&\times\left[\dfrac{f^{\prime}(a)}{\ln q- \ln \rho}+f(a)\right],
	\end{align*}	
	this ends the proof for equation \eqref{6.8}.	
	Next, if we put $n=1$ and $a=1$ we obtain
	
	\begin{align*}
		\left(\dfrac{q}{\rho}\right) I_{\mathcal{R}(\rho,q)} (f_{1})&=I_{\mathcal{R}(\rho,q)}(f)+ \dfrac{(\rho-q)}{\rho^{1}-q^{1}}\dfrac{(\rho^{\rho}-q^{Q})}{\mathcal{R}(\rho^{P},q^{Q})} \sum_{a=0}^{n-1}\left(\dfrac{q}{\rho} \right)^{0}\\
		\nonumber
		&\times\dfrac{f^{\prime}(0)+f(0)\ln\frac{q}{\rho}}{\ln\frac{q}{\rho}}\\
		\nonumber
		&=I_{\mathcal{R}(\rho,q)}(f)+ \dfrac{(\rho^{\rho}-q^{Q})}{\mathcal{R}(\rho^{\rho},q^{Q})} \left[\dfrac{f^{\prime}(0)}{\ln q- \ln \rho}+f(0)\right].	
	\end{align*}
	This completes the proof.
\end{proof}

The $\mathcal{R}(\rho,q)$- Bernoulli polynomials, the Cartitz type can be defined for $a\in \mathbb{Q}$ as follows:

$B_{n;a}(x:\mathcal{R}(\rho,q))$ are defined by the following $\mathcal{R}(\rho,q)$- Volkenborn integral 
$$B_{n;a}(x:\mathcal{R}(\rho,q))=\int_{\BZ_p}\rho^{at}[x+t]_{\mathcal{R}(\rho,q)}^{n}d\mu_{\mathcal{R}(\rho,q)}(t).$$
Setting $\mathcal{R}(\rho,q)=1$, we have the $(\rho,q)-$ Volkenborn calculus
and if we choose $\rho=1$ we obtain equation $\eqref{6.11}$ by \cite{ref52} called the Carlitz's q-Bernoulli polynomials. For $$[x+t]_{\mathcal{R}(\rho,q)}=\rho^{t}[x]_{\mathcal{R}(\rho,q)}+q^{x}[t]_{\mathcal{R}(\mathcal{R}(\rho,q))}.$$
the $\mathcal{R}(\rho,q)$-Bernoulli polynomials is given by:
\begin{eqnarray}\label{6.21}
	B_{n;a}(x:\mathcal{R}(\rho,q))&=&\int_{\BZ_p}\rho^{at}[x+t]_{\mathcal{R}(\rho,q)}^{n}d\mu_{\mathcal{R}(\rho,q)}(t)\\
	\nonumber
	&=&\int_{\BZ_p}\rho^{at}(\rho^{t}[x]_{\mathcal{R}(\rho,q)}+q^{x}[t]_{\mathcal{R}(\rho,q)})^{n}d \mu_{\mathcal{R}(\rho,q)}(t)\\
	\nonumber
	&=&\sum_{r=0}^{n}\left(\begin{array}{c} n \\ r \end{array}\right)[x]^{n-r}_{\mathcal{R}(\rho,q)}q^{rx}\int_{\BZ_{p}} \rho^{(a+n-r)t}[t]^{r}_{\mathcal{R}(\rho,q)}d \mu_{\mathcal{R}(\rho,q)}(t)\\
	\nonumber
	&=&\sum_{r=0}^{n}\left(\begin{array}{c} n \\ r \end{array}\right)[x]^{n-r}_{\mathcal{R}(\rho,q)}q^{rx} B_{k;a+n-r(\mathcal{R}(\rho,q))}	
	\\
	\nonumber
	&=&(q^{x}B_{a}(\mathcal{R}(\rho,q))+[x]_{\mathcal{R}(\rho,q)})^{n}.
\end{eqnarray}
Setting $\mathcal{R}(\rho,q)=1$, we have the $(\rho,q)-$ Volkenborn calculus
and if we choose from equation $\eqref{6.21}$, when $\rho=1$ we obtain $B_{n}(q)=(q^{x}B(q)+[x]_{q})^{n}$. Furthermore, when $\rho=1, \mbox{and} q=1$ we obtain the identity:
$$   B_{n}=(B+x)^{n}=\sum_{r=0}^{n}\left(\begin{array}{c} n \\ r \end{array}\right)B_{r}x^{n-r}.$$

\begin{defn}\label{dfn 4.3}
	Let $n$ be a positive integer. We define the $\mathcal{R}(p,q)$
	\begin{align*}
		\mathcal{A}(x,z:\mathcal{R}(p,q))&=\sum_{n=0}^{\infty}B_{n}(x:\mathcal{R}(p,q))\dfrac{z^{n}}{[n]_{\mathcal{R}(p,q)}!}\\
		\nonumber
		&=\dfrac{z}{e_{\mathcal{R}(p,q)}(z)-1}e_{\mathcal{R}(p,q)}(xz)\quad (|z| <2\pi),
		\\
		\nonumber
		\mathbb{D}(x,z:\mathcal{R}(p,q))&=\sum_{n=0}^{\infty}E_{n}(x:\mathcal{R}(p,q))\dfrac{z^{n}}{[n]_{\mathcal{R}(p,q)}!}\\
		\nonumber
		&=\dfrac{[2]_{\mathcal{R}(p,q)}}{e_{\mathcal{R}(p,q)}(z)+1}e_{\mathcal{R}(p,q)}(xz)\quad (|z|< \pi),\\
		\nonumber
		\mathbb{M}(x,z:\mathcal{R}(p,q))&=\sum_{n=0}^{\infty}G_{n}(x:\mathcal{R}(p,q))\dfrac{z^{n}}{[n]_{\mathcal{R}(p,q)}!}\\
		\nonumber
		&=\dfrac{[2]_{\mathcal{R}(p,q)}z}{e_{\mathcal{R}(p,q)}(z)+1}e_{\mathcal{R}(p,q)}(xz) \quad (|z|< \pi),					
	\end{align*}
	where $B_{n}(x:\mathcal{R}(p,q))$,  $E_{n}(x:\mathcal{R}(p,q))$ and  $G_{n}(x:\mathcal{R}(p,q))$ are the\\ $\mathcal{R}(p,q)-$Bernoulli polynomials, $\mathcal{R}(p,q)-$Euler polynomials and \\$\mathcal{R}(p,q)-$Genocchi polynomials respectively. 
	The case where $x=0$ we obtain 
	$B_{n}(0:\mathcal{R}(p,q))=B_{n}[\mathcal{R}(p,q)$],  $E_{n}(0:\mathcal{R}(p,q))=E_{n}[\mathcal{R}(p,q)$] and  $G_{n}(0:\mathcal{R}(p,q))=G_{n}[\mathcal{R}(p,q)$] which represent $\mathcal{R}(p,q)-$Bernoulli numbers, $\mathcal{R}(p,q)-$Euler numbers and $\mathcal{R}(p,q)-$Genocchi numbers respectively.	
	Also, when we set $\mathcal{R}(p,q)=1$, we obtain the $(p,q)-$analogue in definition \ref{dfn 2.14} respectively.
\end{defn}

\subsection{$p-$adic $\mathcal{R}(\rho,q)-$gamma functions}
Y. Morita's $p-$adic gamma function was used by Hamza Menken and  Özge Çolakoglu \cite{ref100} to consider a $p-$adic analogue of the classical beta function \cite{ref54}. Some fundamental properties of the $p-$adic beta function were discovered, as well as some relationships between the classical beta and $p-$adic beta functions at natural number values. Duran and Acikgoz \cite{ref65} also extended these results to the $(\rho,q)-$gamma and $(\rho,q)-$beta functions. 
For the classical case:
\begin{gather*}
	(n!)_{p}= \prod_{\substack{j<n\\ (p,j)=1}} j
\end{gather*} 
and
\begin{gather*}
	\Gamma_{p}(x)=\lim_{n\rightarrow x}(-1)^{n}\prod_{\substack{j<n\\ (p,j)=1}} j.
\end{gather*}
shall be denoted without loss of generality by:
\begin{gather*}
	(n!)^{p}= \prod_{\substack{j<n\\ (p,j)=1}} j
\end{gather*} 
and
\begin{gather*}
	\Gamma^{p}(x)=\lim_{n\rightarrow x}(-1)^{n}\prod_{\substack{j<n\\ (p,j)=1}} j.
\end{gather*}
for easy notation in the subsequent development.
\bigskip

One can express the $\mathcal{R}(\rho,q)-$gamma function in terms of $p$-adic factorial function $(n!)_{\mathcal{R}(\rho,q)}^{p}$. \\
Now for $n\in\BN$, the $p-$adic $\mathcal{R}(\rho,q)-$factorial function can be written as

\begin{gather}
	(n!)_{\mathcal{R}(\rho,q)}^{p}= \prod_{\substack{j<n\\ (p,j)=1}} [j]_{\mathcal{R}(\rho,q)}.
\end{gather}

\begin{defn}\label{def 6.2}
	Let $\rho$ and $q\in \BC_{p}$ with 
	$$|\rho -1|_{p}<1$$ and  $$|q-1|_{p}<1,$$ 
	$\rho \neq 1, \quad q\neq 1.$ We introduce the $p$-adic $\mathcal{R}(\rho,q)-$factorial function $(x!)^{p}_{\mathcal{R}(\rho,q)}$ as follows
	\begin{gather}
		\Gamma^{p}_{\mathcal{R}(\rho,q)}(x)=\lim_{n\rightarrow x}(-1)^{n}\prod_{\substack{j<n\\ (p,j)=1}} [j]_{\mathcal{R}(\rho,q)}.
	\end{gather}
\end{defn}

\begin{lem}\label{4.5}\cite{ref65}
	For all $x\in \BZ_{p}$ the following results hold;
	\begin{equation*}
		\Gamma^{p}_{\mathcal{R}(\rho,q)}(0)=1, \quad \Gamma^{p}_{\mathcal{R}(\rho,q)}(1)=-1, \quad \mbox{and} \quad |\Gamma^{p}_{\mathcal{R}(\rho,q)}(x)|_{P}=1.
	\end{equation*}
	
\end{lem}
\begin{rem}\label{4.6}
	From the $\mathcal{R}(\rho,q)-$numbers, we have the product rule:	
	$$[kp]_{\mathcal{R}(\rho,q)}=[k]_{\mathcal{R}(\rho^{p},q^{p})} [p]_{\mathcal{R}(\rho^{p},q^{p})}.$$
\end{rem}

\begin{thm}\cite{ref65}
	The following recurrence formula holds true for all $z\in \BZ_{p}$:
	
	\begin{gather}
		\Gamma_{\mathcal{R}(\rho,q)}^{p}(z+1)= \delta^{p}_{\mathcal{R}(\rho,q)} [z]\Gamma_{\mathcal{R}(\rho,q)}^{p}(z)
	\end{gather}
	where
	\begin{eqnarray*}
		\delta^{p}_{\mathcal{R}(\rho,q)} [z]= \left\{\begin{array}{lr} -[z]_{\mathcal{R}(\rho,q)} \quad \mbox{if   } \quad |z|_{p}=0, \quad \\
			-1 \quad \quad \quad \quad \mbox{if }   \quad |z|_{p}<
			1. \quad \end{array} \right.
	\end{eqnarray*}
	
\end{thm}

\begin{thm}
	The following recurrence formula holds true for all $n\in \BN$:
	
	\begin{gather}
		\Gamma_{\mathcal{R}(\rho,q)}^{p}(n+1)= (-1)^{n+1}\dfrac{[n]_{\mathcal{R}(\rho,q)}!}{[p]_{\mathcal{R}(\rho,q)}^{\left \lfloor \dfrac{n}{p}\right \rfloor}\left[\left \lfloor \dfrac{n}{p}\right \rfloor \right]_{\mathcal{R}(\rho^{p},q^{p})}!},
	\end{gather}
	where $\lfloor \cdot \rfloor$ is the greatest integer function.
\end{thm}
\begin{proof}
	We observe from definition \ref{def 6.2} that 
	\begin{align}
		\Gamma^{p}_{\mathcal{R}(\rho,q)}(n+1)&=(-1)^{n+1}\prod_{\substack{j<n\\ (p,j)=1}} [j]_{\mathcal{R}(\rho,q)}\\
		\nonumber
		&=(-1)^{n+1}\dfrac{[1]_{\mathcal{R}(\rho,q)}[2]_{\mathcal{R}(\rho,q)} \cdots [n]_{\mathcal{R}(\rho,q)}}{[p]_{\mathcal{R}(\rho,q)}[2p]_{\mathcal{R}(\rho,q)}\cdots \left[\left \lfloor \dfrac{n}{p}\right \rfloor \right]_{\mathcal{R}(\rho,q)}}\\
		\nonumber
		&=(-1)^{n+1}\dfrac{[n]_{\mathcal{R}(\rho,q)}!}{[p]_{\mathcal{R}(\rho,q)}^{\left \lfloor \dfrac{n}{p}\right \rfloor}[1]_{\mathcal{R}(\rho^{p},q^{p})}[2]_{\mathcal{R}(\rho^{p},q^{p})} \cdots\left[\left \lfloor \dfrac{n}{p}\right \rfloor \right]_{\mathcal{R}(\rho^{p},q^{p})}!}	
	\end{align}	
	by the application of remark \ref{4.6} we obtain the proof.	
	
\end{proof}

\begin{lem}
	Let $m_{n}$ be the sum of digits of $n=\sum_{j=0}^{m}a_{j}p^{j}$ ($a_{m}\neq0$) in base $p$. Then 
	
	\begin{align}\label{lem 6.9}
		\left[\left \lfloor \dfrac{n}{p}\right \rfloor \right]_{\mathcal{R}(\rho^{p},q^{p})}!&=(-1)^{n+1-m}\left(-[p]_{\mathcal{R}(\rho^{p},q^{p})}\right)^\frac{{(n-m_{n})}}{(p-1)}\\
		\nonumber
		&\times \prod_{j=0}^{m-1} \dfrac{	\left[\left \lfloor \dfrac{n}{p^{j+1}}\right \rfloor \right]_{\mathcal{R}(\rho^{p},q^{p})}!}{\left[\left \lfloor \dfrac{n}{p^{j}}\right \rfloor \right]_{\mathcal{R}(\rho^{p},q^{p})}!} \prod_{i=0}^{m}\Gamma_{\mathcal{R}(\rho,q)}^{p} \left(\left[\left \lfloor \dfrac{n}{p}\right \rfloor \right]+1 \right) 
	\end{align}
and	
	\begin{align}\label{lem 6.10}
		[n]_{\mathcal{R}(\rho^{p},q^{p})}!&=(-1)^{n+1-m}\left(-[p]_{\mathcal{R}(\rho^{p},q^{p})}\right)^\frac{{(n-m_{n})}}{(p-1)}\left[\left \lfloor \dfrac{n}{p}\right \rfloor \right]_{\mathcal{R}(\rho^{p},q^{p})}! \\
		\nonumber
		&\times \prod_{j=0}^{m-1} \dfrac{	\left[\left \lfloor \dfrac{n}{p^{j+1}}\right \rfloor \right]_{\mathcal{R}(\rho^{p},q^{p})}!}{\left[\left \lfloor \dfrac{n}{p^{j}}\right \rfloor \right]_{\mathcal{R}(\rho^{p},q^{p})}!}\prod_{i=0}^{m}\Gamma_{\mathcal{R}(\rho,q)}^{p} \left(\left[\left \lfloor \dfrac{n}{p^{i}}\right \rfloor \right]+1 \right)		
	\end{align}
\end{lem}
\begin{proof}
	\begin{align*}
		[n]_{\mathcal{R}(\rho,q)}! &= (-1)^{n+1} [p]_{\mathcal{R}(\rho,q)}^{\left \lfloor \dfrac{n}{p}\right \rfloor}\left[\left \lfloor \dfrac{n}{p}\right \rfloor \right]_{\mathcal{R}(\rho^{p},q^{p})}! 	\Gamma^{p}_{\mathcal{R}(\rho,q)}(n+1)
		\\
		\nonumber
		\left[\left \lfloor \dfrac{n}{p^{0}}\right \rfloor \right]_{\mathcal{R}(\rho,q)}!&= (-1)^{\left \lfloor \dfrac{n}{p^{0}}\right \rfloor+1} [p]_{\mathcal{R}(\rho,q)}^{\left \lfloor \dfrac{n}{p^{1}}\right \rfloor}\left[\left \lfloor \dfrac{n}{p^{2}}\right \rfloor \right]_{\mathcal{R}(\rho^{p},q^{p})}! \\
		\nonumber	
		&\times	\Gamma^{p}_{\mathcal{R}(\rho,q)}\left(\left \lfloor \dfrac{n}{p^{0}}\right \rfloor+1 \right),
		\\
		\nonumber
		\left[\left \lfloor \dfrac{n}{p^{1}}\right \rfloor \right]_{\mathcal{R}(\rho,q)}!&= (-1)^{\left \lfloor \dfrac{n}{p^{1}}\right \rfloor+1} [p]_{\mathcal{R}(\rho,q)}^{\left \lfloor \dfrac{n}{p^{2}}\right \rfloor}\left[\left \lfloor \dfrac{n}{p^{2}}\right \rfloor \right]_{\mathcal{R}(\rho^{p},q^{p})}! \\
		\nonumber	
		&\times \Gamma^{p}_{\mathcal{R}(\rho,q)}\left(\left \lfloor \dfrac{n}{p^{1}}\right \rfloor+1 \right),
		\\
		&\vdots
		\\
		\left[\left \lfloor \dfrac{n}{p^{m}}\right \rfloor \right]_{\mathcal{R}(\rho,q)}!&= (-1)^{\left \lfloor \dfrac{n}{p^{m}}\right \rfloor+1} [p]_{\mathcal{R}(\rho,q)}^{\left \lfloor \dfrac{n}{p^{m+1}}\right \rfloor}\left[\left \lfloor \dfrac{n}{p^{m+1}}\right \rfloor \right]_{\mathcal{R}(\rho^{p},q^{p})}! \\
		\nonumber
		&\times \Gamma^{p}_{\mathcal{R}(\rho,q)}\left(\left \lfloor \dfrac{n}{p^{m}}\right \rfloor+1 \right),
	\end{align*}
	multiplying the inductive process above we obtain:
	\begin{align*}
		\left[\left \lfloor \dfrac{n}{p^{m}}\right \rfloor \right]_{\mathcal{R}(\rho,q)}!&= (-1)^{\left \lfloor \dfrac{n}{p^{0}}\right \rfloor+\left \lfloor \dfrac{n}{p^{1}}\right \rfloor +\cdots + \left \lfloor \dfrac{n}{p^{m}}\right \rfloor +m+1} \\
		&\times
		[p]_{\mathcal{R}(\rho,q)}^{\left \lfloor \dfrac{n}{p^{0}}\right \rfloor +\left \lfloor \dfrac{n}{p^{1}}\right \rfloor +\cdots + \left \lfloor \dfrac{n}{p^{m+1}}\right \rfloor} \\	&\times
		\left[\left \lfloor \dfrac{n}{p^{m+1}}\right \rfloor \right]_{\mathcal{R}(\rho^{p},q^{p})}!
		\prod_{j=0}^{m-1} \dfrac{	\left[\left \lfloor \dfrac{n}{p^{j+1}}\right \rfloor \right]_{\mathcal{R}(\rho^{p},q^{p})}!}{\left[\left \lfloor \dfrac{n}{p^{j}}\right \rfloor \right]_{\mathcal{R}(\rho^{p},q^{p})}!} \\
		&\times\prod_{i=0}^{m} 	\Gamma^{p}_{\mathcal{R}(\rho,q)}\left(\left \lfloor \dfrac{n}{p^{i}}\right \rfloor+1 \right)
	\end{align*}
	by simple computations we arrive at the equation \eqref{lem 6.9}.\\
	Next for the equation \eqref{lem 6.10} 
	
	\begin{align*}
		[n]_{\mathcal{R}(\rho,q)}!&= (-1)^{\left \lfloor \dfrac{n}{p^{0}}\right \rfloor+\left \lfloor \dfrac{n}{p^{1}}\right \rfloor +\cdots + \left \lfloor \dfrac{n}{p^{m}}\right \rfloor +m+1} [p]_{\mathcal{R}(\rho,q)}^{\left \lfloor \dfrac{n}{p^{0}}\right \rfloor+\left \lfloor \dfrac{n}{p^{1}}\right \rfloor +\cdots + \left \lfloor \dfrac{n}{p^{m+1}}\right \rfloor} \\
		&\times
		\left[\left \lfloor \dfrac{n}{p}\right \rfloor \right]_{\mathcal{R}(\rho^{p},q^{p})}! \prod_{j=0}^{m-1} \dfrac{	\left[\left \lfloor \dfrac{n}{p^{j+1}}\right \rfloor \right]_{\mathcal{R}(\rho^{p},q^{p})}!}{\left[\left \lfloor \dfrac{n}{p^{j}}\right \rfloor \right]_{\mathcal{R}(\rho^{p},q^{p})}!}\prod_{i=0}^{m}	\Gamma^{p}_{\mathcal{R}(\rho,q)}\left(\left \lfloor \dfrac{n}{p^{i}}\right \rfloor+1 \right)\\
		&=(-1)^\frac{{(n-m_{n})}}{(p-1)}(-1)^{n+1-m}\left([p]_{\mathcal{R}(\rho^{p},q^{p})}\right)^\frac{{(n-m_{n})}}{(p-1)}\left[\left \lfloor 
		\dfrac{n}{p}\right \rfloor \right]_{\mathcal{R}(\rho^{p},q^{p})}!\\
		\nonumber
		&\times \prod_{j=0}^{m-1} \dfrac{\left[\left \lfloor \dfrac{n}{p^{j+1}}\right \rfloor \right]_{\mathcal{R}(\rho^{p},q^{p})}!}{\left[\left \lfloor \dfrac{n}{p^{j}}\right \rfloor \right]_{\mathcal{R}(\rho^{p},q^{p})}!}
		\times \prod_{i=0}^{m}\Gamma_{\mathcal{R}(\rho,q)}^{p} \left(\left[\left \lfloor \dfrac{n}{p^{i}}\right \rfloor \right]+1 \right).		
	\end{align*} 
	This finishes the proof.
\end{proof}

\begin{thm}
	For a prime number $p$ and $m_{n}$ be the sum of digits of $n=\sum_{j=o}^{m} a_{j}p^{j}$ in base $p$ where $n\in \BN$. For $0\leq k \leq m$	 and $j=0,1,\cdots,m$ then the following identity holds: 
	
	\begin{gather}
		\dfrac{\left[\left \lfloor \dfrac{n}{p^{j}}\right \rfloor \right]_{\mathcal{R}(\rho,q)}!}{[p]_{\mathcal{R}(\rho,q)}^{\left \lfloor \dfrac{n}{p^{j}}\right \rfloor}\left[\left \lfloor \dfrac{n}{p^{j}}\right \rfloor \right]_{\mathcal{R}(\rho^{p},q^{p})}!}=\prod_{k=1}^{\left \lfloor \dfrac{n}{p^{j}}\right \rfloor}\dfrac{\rho^{k}- q^{k}}{\rho^{kp}-q^{kp}}.
	\end{gather}
	Consequently we obtain;
	
	\begin{align}
		[n]_{\mathcal{R}(\rho^{p},q^{p})}!&=(-1)^{\frac{n-m_{n}}{(p-1)+n+1-m}}\prod_{k=1}^{\left \lfloor \dfrac{n}{p^{1}}\right \rfloor}\dfrac{\rho^{k}- q^{k}}{\rho^{kp}-q^{kp}} \cdots \prod_{k=1}^{\left \lfloor \dfrac{n}{p^{m}}\right \rfloor }\dfrac{\rho^{k}- q^{k}}{\rho^{kp}-q^{kp}}	\nonumber\\
		&\times \prod_{i=0}^{m}\Gamma_{\mathcal{R}(\rho,q)}^{p} \left(\left \lfloor \dfrac{n}{p^{j}}\right \rfloor +1 \right).		
	\end{align}
\end{thm}
\begin{proof}
	\begin{align*}
		\dfrac{\left[\left \lfloor \dfrac{n}{p^{j}}\right \rfloor \right]_{\mathcal{R}(\rho,q)}!}{[p]_{\mathcal{R}(\rho,q)}^{\left \lfloor \dfrac{n}{p^{j}}\right \rfloor}\left[\left \lfloor \dfrac{n}{p^{j}}\right \rfloor \right]_{\mathcal{R}(\rho^{p},q^{p})}!}&=\dfrac{[1]_{\mathcal{R}(p,q)}[2]_{\mathcal{R}(\rho,q)} \cdots \left[\left \lfloor \dfrac{n}{p^{j}}\right \rfloor \right]_{\mathcal{R}(\rho^{p},q^{p^{j}})}!}{[p]_{\mathcal{R}(\rho,q)}^{\left \lfloor \dfrac{n}{p}\right \rfloor}[1]_{\mathcal{R}(\rho^{p},q^{p})}[2]_{\mathcal{R}(\rho^{p},q^{p})} \cdots\left[\left \lfloor \dfrac{n}{p^{j}}\right \rfloor \right]_{\mathcal{R}(\rho^{p},q^{p})}!}\\
		&= \dfrac{\dfrac{\rho- q}{\rho-q}\dfrac{\rho^{2}- q^{2}}{\rho-q}\cdots \dfrac{\rho^{\left \lfloor \dfrac{n}{p^{j}}\right \rfloor }- q^{\left \lfloor \dfrac{n}{p^{j}}\right \rfloor }}{\rho-q} }{\dfrac{\rho^{p}- q^{p}}{\rho^{p}-q^{p}}\dfrac{\rho^{2p}- q^{2p}}{\rho^{p}-q^{p}}\cdots \dfrac{\rho^{\left \lfloor \dfrac{n}{p^{j}}\right \rfloor p }- q^{\left \lfloor \dfrac{n}{p^{j}}\right \rfloor p }}{\rho-q}}\\
		&=\dfrac{\mathcal{R}(\rho-q)\mathcal{R}(\rho^{2}-q^{2})}{\mathcal{R}(\rho^{p}-q^{p})\mathcal{R}(\rho^{2p}-q^{2p})}
		\\
		& \times \dfrac{\cdots \mathcal{R}\left(\rho^{\left \lfloor \dfrac{n}{p^{j}}\right \rfloor }- q^{\left \lfloor  \dfrac{n}{p^{j}}\right \rfloor }\right)}{\cdots \mathcal{R}\left(\rho^{\left \lfloor \dfrac{n}{p^{j}}\right \rfloor p }- q^{\left \lfloor  \dfrac{n}{p^{j}}\right \rfloor p }\right)}	
	\end{align*}
	for $0\leq j \leq m$. 
	\\
	Next for the $[n]_{\mathcal{R}(\rho^{p},q^{p})}!$ one can easily arrive at 
	
	\begin{align*}
		[n]_{\mathcal{R}(\rho^{p},q^{p})}!&=(-1)^{\frac{n-m_{n}}{(p-1)+n+1-m}}\prod_{k=1}^{\left \lfloor \dfrac{n}{p^{1}}\right \rfloor}\dfrac{\rho^{k}- q^{k}}{\rho^{kp}-q^{kp}} \cdots \prod_{k=1}^{\left \lfloor \dfrac{n}{p^{m}}\right \rfloor }\dfrac{\rho^{k}- q^{k}}{\rho^{kp}-q^{kp}}\\
		\nonumber
		&\times \prod_{i=0}^{m}\Gamma_{\mathcal{R}(\rho,q)}^{p} \left(\left \lfloor \dfrac{n}{p^{j}}\right \rfloor +1 \right),		
	\end{align*}
	thus the proof is completed.
\end{proof}
\subsection{$p-$adic $\mathcal{R}(p,q)-$beta functions}
The $p-$adic beta function $\beta_{p}:\BZ_{p} \times \BZ_{p} \longrightarrow \mathbb{Q}_{p}$ is given by: 
\begin{gather}
	\beta_{p}(x+y)=\dfrac{\Gamma_{p}(x) \Gamma_{p}(y) }{\Gamma_{p}(x+y)}
\end{gather}
for $x,y\in \BZ_{p}.$
For the purpose of notation we shall denote the $p-$adic beta function as follows:
\begin{gather}
	\beta^{p}(x+y)=\dfrac{\Gamma^{p}(x) \Gamma^{p}(y) }{\Gamma^{p}(x+y)}.
\end{gather}

\begin{defn}\label{def6.9}
	Let $\rho$ and $q\in \BC_{p}$ with 
	$|\rho -1|_{p}<1$ and  $|q-1|_{p}<1,$
	$\rho\neq1, \quad q\neq 1.$ We define the $p-$adic $\mathcal{R}(\rho,q)-$beta function via the $p-$adic $\mathcal{R}(\rho,q)-$gamma functions as follows:
	\begin{gather}
		\beta_{\mathcal{R}(\rho,q)}^{p}(x,y)=\dfrac{\Gamma_{\mathcal{R}(\rho,q)}^{p}(x) \Gamma_{\mathcal{R}(\rho,q)}^{p}(y) }{\Gamma_{\mathcal{R}(\rho,q)}^{p}(x+y)}
	\end{gather}
	for $x,y\in \BZ_{p}.$
\end{defn}

\begin{thm} 	
	The $\mathcal{R}(\rho,q)-$beta functions have the following properties:
	\begin{enumerate}
		\item [(i)] $\beta_{\mathcal{R}(\rho,q)}^{p}(x,y+1)=	\dfrac{\delta^{p}_{\mathcal{R}(\rho,q)}(y)}{	\delta^{p}_{\mathcal{R}(\rho,q)}(x+y)}\beta_{\mathcal{R}(\rho,q)}^{p}(x,y)$
		\\
		\item [(ii)] $\beta_{\mathcal{R}(\rho,q)}^{p}(x+1,y)=\dfrac{\delta^{p}_{\mathcal{R}(\rho,q)}(x)}{\delta^{p}_{\mathcal{R}(\rho,q)}(x+y)}\beta_{\mathcal{R}(\rho,q)}^{p}(x,y)$
		\\
		\item [(iii)] $\beta_{\mathcal{R}(\rho,q)}^{p}(x+1,y)=\dfrac{\delta^{p}_{\mathcal{R}(\rho,q)}(x)}{\delta^{p}_{\mathcal{R}(\rho,q)}(y)}\beta_{\mathcal{R}(\rho,q)}^{p}(x,y+1)$
		\item [(v)] $\beta_{\mathcal{R}(\rho,q)}^{p}(x+1,y)+\beta_{\mathcal{R}(\rho,q)}^{p}(x,y+1)=\dfrac{\delta^{p}_{\mathcal{R}(\rho,q)}(x) + \delta^{p}_{\mathcal{R}(\rho,q)}(y)}{\delta^{p}_{\mathcal{R}(\rho,q)}(x+y)}\beta_{\mathcal{R}(\rho,q)}^{p}(x,y)$
		\\
		\item[(vi)] $\beta_{\mathcal{R}(\rho,q)}^{p}(x+1,y+1)=\dfrac{\delta^{p}_{\mathcal{R}(\rho,q)}(x)+ \delta^{p}_{\mathcal{R}(\rho,q)}(y)}{\delta^{p}_{\mathcal{R}(\rho,q)}(x+y+1)\delta^{p}_{\mathcal{R}(\rho,q)}(x+y)}\\
		\times \beta_{\mathcal{R}(\rho,q)}^{p}(x,y)$
		\\
		\item [(vii)] $\beta_{\mathcal{R}(\rho,q)}^{p}(x,y)+\beta_{\mathcal{R}(\rho,q)}^{p}(x+y,z)+\beta_{\mathcal{R}(\rho,q)}^{p}(x+y+z,w)\\ \nonumber
		=\dfrac{\Gamma_{\mathcal{R}(\rho,q)}^{p}(x) \Gamma_{\mathcal{R}(\rho,q)}^{p}(y)\Gamma_{\mathcal{R}(\rho,q)}^{p}(z)\Gamma_{\mathcal{R}(\rho,q)}^{p}(w) }{\Gamma_{\mathcal{R}(\rho,q)}^{p}(x+y+z+w)}$
		\\
		\item [(viii)] $\beta_{\mathcal{R}(\rho,q)}^{p}(x,1-x)=-\Gamma_{\mathcal{R}(\rho,q)}^{p}{(x)}\Gamma_{\mathcal{R}(\rho,q)}^{p}(1-x).$
	\end{enumerate}
\end{thm}

\begin{proof}
	For $(i)$ we have 
	\begin{align*}
		\beta_{\mathcal{R}(\rho,q)}^{p}(x,y+1)&=\dfrac{\Gamma_{\mathcal{R}(\rho,q)}^{p}(x) \Gamma_{\mathcal{R}(\rho,q)}^{p}(y+1) }{\Gamma_{\mathcal{R}(\rho,q)}^{p}(x+y+1)}\\
		&=\dfrac{\Gamma_{\mathcal{R}(\rho,q)}^{p}(x)\delta^{p}_{\mathcal{R}(\rho,q)}(y) \Gamma_{\mathcal{R}(\rho,q)}^{p}(y) }{\delta^{p}_{\mathcal{R}(\rho,q)}(x+y)\Gamma_{\mathcal{R}(\rho,q)}^{p}(x+y)}\\
		&=\dfrac{\delta^{p}_{\mathcal{R}(\rho,q)}(y)}{\delta^{p}_{\mathcal{R}(\rho,q)}
			(x+y)}\left( \dfrac{\Gamma_{\mathcal{R}(\rho,q)}^{p}(x)\Gamma_{\mathcal{R}(\rho,q)}^{p}(y)}{\Gamma_{\mathcal{R}(\rho,q)}^{p}(x+y)}\right)\\
		&=\dfrac{\delta^{p}_{\mathcal{R}(\rho,q)}(y)}{	\delta^{p}_{\mathcal{R}(\rho,q)}(x+y)}\beta_{\mathcal{R}(\rho,q)}^{p}(x,y),
	\end{align*}
	Similarly for $(ii)$ we get
	\begin{align*}
		\beta_{\mathcal{R}(\rho,q)}^{p}(x+1,y)&=\dfrac{\Gamma_{\mathcal{R}(\rho,q)}^{p}(x+1) \Gamma_{\mathcal{R}(\rho,q)}^{p}(y) }{\Gamma_{\mathcal{R}(\rho,q)}^{p}(x+y+1)}\\
		&=\dfrac{\delta^{p}_{\mathcal{R}(\rho,q)}[x]_{\mathcal{R}(\rho,q)} \Gamma_{\mathcal{R}(\rho,q)}^{p}(x)\Gamma_{\mathcal{R}(\rho,q)}^{p}(y) }{\delta^{p}_{\mathcal{R}(\rho,q)}[x+y]_{\mathcal{R}(\rho,q)}\Gamma_{\mathcal{R}(\rho,q)}^{p}(x+y)}\\
		&=\dfrac{\delta^{p}_{\mathcal{R}(\rho,q)}(x)}{\delta^{p}_{\mathcal{R}(\rho,q)}(x+y)}\left( \dfrac{\Gamma_{\mathcal{R}(\rho,q)}^{p}(x)\Gamma_{\mathcal{R}(\rho,q)}^{p}(y)}{\Gamma_{\mathcal{R}(\rho,q)}^{p}(x+y)}\right)\\
		&=\dfrac{\delta^{p}_{\mathcal{R}(\rho,q)}(x)}{	\delta^{p}_{\mathcal{R}(\rho,q)}(x+y)}\beta_{\mathcal{R}(\rho,q)}^{p}(x,y),
	\end{align*}
	furthermore, one can simplify;
	
	\begin{align*}
		\beta_{\mathcal{R}(\rho,q)}^{p}(x+1,y)&=\dfrac{\delta^{p}_{\mathcal{R}(\rho,q)}(x)}{	\delta^{p}_{\mathcal{R}(\rho,q)}(x+y)}\beta_{\mathcal{R}(\rho,q)}^{p}(x,y)
		\\
		&=\dfrac{\delta^{p}_{\mathcal{R}(\rho,q)}(x)}{\delta^{p}_{\mathcal{R}(\rho,q)}(y)}\left(\dfrac{\delta^{p}_{\mathcal{R}(\rho,q)}(y)}{	\delta^{p}_{\mathcal{R}(\rho,q)}(x+y)}\beta_{\mathcal{R}(\rho,q)}^{p}(x,y)\right)
		\\
		&=\dfrac{\delta^{p}_{\mathcal{R}(\rho,q)}(x)}{\delta^{p}_{\mathcal{R}(\rho,q)}(y)}\beta_{\mathcal{R}(\rho,q)}^{p}(x,y+1).
	\end{align*}
	The proof of $(v)$ follows from $(i)$ and $(ii)$ easily. Also, proof of $(vi)$ is trivial.
	For $(vii)$ we have;
	
	\begin{align*}
		&\beta_{\mathcal{R}(\rho,q)}^{p}(x,y)+\beta_{\mathcal{R}(\rho,q)}^{p}(x+y,z)+\beta_{\mathcal{R}(\rho,q)}^{p}(x+y+z,w)\\
		&=\dfrac{\Gamma_{\mathcal{R}(\rho,q)}^{p}(x)\Gamma_{\mathcal{R}(\rho,q)}^{p}(y)}{\Gamma_{\mathcal{R}(\rho,q)}^{p}(x+y)}\dfrac{\Gamma_{\mathcal{R}(\rho,q)}^{p}(x+y)\Gamma_{\mathcal{R}(\rho,q)}^{p}(z)}{\Gamma_{\mathcal{R}(\rho,q)}^{p}(x+y+z)}\\
		&\times	\dfrac{\Gamma_{\mathcal{R}(\rho,q)}^{p}(x+y+z)\Gamma_{\mathcal{R}(\rho,q)}^{p}(w)}{\Gamma_{\mathcal{R}(\rho,q)}^{p}(x+y+z+w)}\\
		&=\dfrac{\Gamma_{\mathcal{R}(\rho,q)}^{p}(x)\Gamma_{\mathcal{R}(\rho,q)}^{p}(y) \Gamma_{\mathcal{R}(\rho,q)}^{p}(x+y)\Gamma_{\mathcal{R}(\rho,q)}^{p}(z) }{\Gamma_{\mathcal{R}(\rho,q)}^{p}(x+y)\Gamma_{\mathcal{R}(\rho,q)}^{p}(x+y+z)}\\
		&\times \dfrac{\Gamma_{\mathcal{R}(\rho,q)}^{p}(x+y+z)\Gamma_{\mathcal{R}(\rho,q)}^{p}(z)\Gamma_{\mathcal{R}(\rho,q)}^{p}(w)}{\Gamma_{\mathcal{R}(\rho,q)}^{p}(x+y+z+w)}\\
		&=\dfrac{\Gamma_{\mathcal{R}(\rho,q)}^{p}(x) \Gamma_{\mathcal{R}(\rho,q)}^{p}(y)\Gamma_{\mathcal{R}(\rho,q)}^{p}(z)\Gamma_{\mathcal{R}(\rho,q)}^{p}(w) }{\Gamma_{\mathcal{R}(\rho,q)}^{p}(x+y+z+w)}
	\end{align*}
	Finally, from definition \ref{def6.9} we have;
	\begin{align*}
		\beta_{\mathcal{R}(\rho,q)}^{p}(x,1-x)&=\dfrac{\Gamma_{\mathcal{R}(\rho,q)}^{p}(x) \Gamma_{\mathcal{R}(\rho,q)}^{p}(1-x) }{\Gamma_{\mathcal{R}(\rho,q)}^{p}(x+1-x)}\\&=\dfrac{\Gamma_{\mathcal{R}(\rho,q)}^{p}(x) \Gamma_{\mathcal{R}(\rho,q)}^{p}(1-x) }{\Gamma_{\mathcal{R}(\rho,q)}^{p}(1)}\\
		&=-\Gamma_{\mathcal{R}(\rho,q)}^{p}{(x)}\Gamma_{\mathcal{R}(\rho,q)}^{p}(1-x),
	\end{align*}
	which concludes the proof using lemma \ref{4.5}.
\end{proof}

\section{Concluding remarks}\label{rem}
In this paper, we presented the Hounkonnou $\textit{et al.}$ $\mathcal{R}(p,q)-$exponential functions and  $\mathcal{R}(p,q)-$trigonometric functions in a formal way. The  $\mathcal{R}(p,q)-$ analogue of the Euler and Bernoulli polynomials and their numbers were established using these functions. In addition to defining the $\mathcal{R}(p,q)$-analogue for the Euler-zigzag numbers, we also extended the $\mathcal{R}(p,q)-$integration to the beta and gamma functions and offered immediate consequences like the integration by part, as well as different aspects of the gamma and beta functions, by applying the $\mathcal{R}(p,q)-$derivatives established.
Furthermore, we defined the Hounkounnou $\textit{et al.}$ $\mathcal{R}(p,q)-$power basis and, together with it, introduced significant concepts such as the $\mathcal{R}(p,q)-$Taylor expansion and many others.
Then, we established the definition of a $p-$adic fermion spin Lie group and its corresponding Lie algebra, and we showed that the $S_{+}$, $S_{-}$, and $S_{z}$ generate the Iwasawa Lie algebra.
The principal congruence subgroups and the $p$-adic fermion Lie algebra are related. We focused on this relation and demonstrated that the $p-$adic zeta integral of an elementary particle such as the fermionic $p-$adic spin Lie algebra ($\spin_{\BZ_{p}}(\frac{1}{2})$) is unique up to equivalence, also, we showed that every fermionic ghost polynomial; $\GO_{2l+1}$, $\GSp_{2l}$ and $\GO_{2l}$ of type $\Pi (B_{l})$, $\Pi (C_{l})$ and $\Pi( D_{l})$ has a friendly ghost.
The Volkerborn, Bernoulli, and Genocchi polynomials, as well as their corresponding $\mathcal{R}(p,q)-$analogues, are some features of the T. Kim \cite{ref62} that we also took into consideration.
Finally, the $\mathcal{R}(p,q)-$beta function and the $\mathcal{R}(p,q)-$gamma function are then extended to the $p-$adic $\mathcal{R}(\rho,q)-$beta function and $\mathcal{R}(\rho,q)-$gamma function, respectively.
\bigskip
\subsection*{Acknowledgment}
This work is supported by the NLAGA-SIMONS grant.
\\
 The ICMPA - UNESCO chair is in partnership with Daniel Lagolnitzer Foundation(DIF), France, and the Association pour la Promotion Scientifique de l'Afrique (APSA), supporting the development of Mathematical Physics in Africa. The authors are grateful to anonymous referees for careful reading of the manuscript and helpful comments.


%

\end{document}